\documentclass[12pt,reqno,a4paper]{amsart}
\pdfoutput=1

\usepackage[colorlinks,plainpages,hypertexnames=false,plainpages=false]{hyperref}
\hypersetup{urlcolor=blue, citecolor=blue, linkcolor=blue}
\usepackage{amsmath}
\usepackage{amsthm}
\usepackage{amssymb}
\usepackage{amsfonts}
\usepackage{enumitem}
\usepackage{stmaryrd}
\usepackage{todonotes}
\usepackage[noend]{algorithmic} 

\usepackage{array}
\newcolumntype{L}[1]{>{\raggedright\let\newline\\\arraybackslash\hspace{0pt}}m{#1}}
\newcolumntype{C}[1]{>{\centering\let\newline\\\arraybackslash\hspace{0pt}}m{#1}}
\newcolumntype{R}[1]{>{\raggedleft\let\newline\\\arraybackslash\hspace{0pt}}m{#1}}
\usepackage{texdraw}
\usepackage{tikz-cd}

\usepackage{tikz}
\usetikzlibrary{arrows}
\usetikzlibrary{3d}
\usetikzlibrary{decorations.pathreplacing}
\usetikzlibrary{decorations.pathmorphing}
\usetikzlibrary{matrix}
\usetikzlibrary{calc}
\usetikzlibrary{shapes}
\usetikzlibrary{patterns}
\usetikzlibrary{fit,backgrounds,scopes}

\newtheoremstyle{theoremstyle}
{10pt}      
{5pt}       
{\itshape}  
{}          
{\bfseries} 
{}         
{ }        
{}          

\newtheoremstyle{algorithmstyle}
{10pt}      
{5pt}       
{}  
{}          
{\bfseries} 
{}         
{ }      
{}          

\newtheoremstyle{examplestyle}
{10pt}      
{5pt}       
{}          
{}          
{\bfseries} 
{}         
{ }      
{}          

\theoremstyle{theoremstyle}
\newtheorem{theorem}{Theorem}[section]
\newtheorem*{theorem*}{Theorem}
\newtheorem{lemma}[theorem]{Lemma}
\newtheorem{proposition}[theorem]{Proposition}
\newtheorem*{proposition*}{Proposition}

\newtheorem*{corollary*}{Corollary}

\theoremstyle{examplestyle}

\newtheorem{definition}[theorem]{Definition}
\newtheorem{definition*}{Definition}

\newtheorem{remark}[theorem]{Remark}
\newtheorem{remark*}{Remark}

\theoremstyle{algorithmstyle}

\newcommand{\RR }{\mathbb{R}}

\newcommand{\ZZ }{\mathbb{Z}}

\newcommand{\PP }{\mathbb{P}}

\newcommand{\suchthat}{\;\ifnum\currentgrouptype=16 \middle\fi|\;}

\def\rh{{rh}}
\def\nrh{{non-rh}}

\DeclareMathOperator{\valuation}{val}

\DeclareMathOperator{\Trop}{Trop}
\DeclareMathOperator{\trop}{trop}

\DeclareMathOperator{\ft}{ft}

\DeclareMathOperator{\hyp}{hyp}
\DeclareMathOperator{\an}{an}

\DeclareMathOperator{\conv}{conv}

\renewcommand{\emph}[1]{\textit{\textcolor{red}{#1}}}

\DeclareRobustCommand{\tnullnullnull}[1]
{%
  \begin{tikzpicture}[scale=0.75,#1]
    \draw (0,0) -- (1.16ex,2ex) -- (2.32ex,0) -- cycle
    (0,0) -- (1.16ex,0.58ex)
    (2.32ex,0) -- (1.16ex,0.58ex)
    (1.16ex,2ex) -- (1.16ex,0.58ex);
  \end{tikzpicture}
}

\DeclareRobustCommand{\tnulltwonull}[1]
{%
  \begin{tikzpicture}[scale=0.75,#1]
    \draw (0,0) to[out=130,in=230] (0,2ex)
    (0,0) to[out=50,in=310] (0,2ex);
    \draw (3ex,0) to[out=130,in=230] (3ex,2ex)
    (3ex,0) to[out=50,in=310] (3ex,2ex);
    \draw (0,0) -- (3ex,0)
    (0,2ex) -- (3ex,2ex);
  \end{tikzpicture}
}

\DeclareRobustCommand{\toneoneone}[1]
{%
  \begin{tikzpicture}[scale=0.75,#1]
    \draw (0,0) to[out=130,in=230] (0,2ex)
    (0,0) to[out=50,in=310] (0,2ex);
    \draw (0,2ex) -- (2ex,1ex)
    (0,0) -- (2ex,1ex)
    (2ex,1ex) -- (3ex,1ex);
    \draw (3.5ex,1ex) circle (0.5ex);
  \end{tikzpicture}
}

\DeclareRobustCommand{\ttwoonetwo}[1]
{%
  \begin{tikzpicture}[scale=0.75,#1]
    \draw (0,0) circle (0.5ex);
    \draw (0.5ex,0) -- (2ex,0);
    \draw (2ex,0) to[out=80,in=100,min distance=1ex] (3.5ex,0);
    \draw (2ex,0) to[out=280,in=260,min distance=1ex] (3.5ex,0);
    \draw (3.5ex,0) -- (5ex,0);
    \draw (5.5ex,0) circle (0.5ex);
  \end{tikzpicture}
}

\DeclareRobustCommand{\tthreenullthree}[1]
{%
  \begin{tikzpicture}[scale=0.75]
    \draw (0,0) -- (30:0.6ex);
    \draw (0,0) -- (150:0.6ex);
    \draw (0,0) -- (270:0.6ex);
    \draw (30:0.9ex) circle (0.3ex);
    \draw (150:0.9ex) circle (0.3ex);
    \draw (270:0.9ex) circle (0.3ex);
  \end{tikzpicture}
}

\definecolor{edgeRed}{RGB}{255,0,0}
\definecolor{edgeOrange}{RGB}{255,148,0}
\definecolor{edgeYellow}{RGB}{253,229,65}
\definecolor{edgeGreen}{RGB}{15,173,0}
\definecolor{edgeBlue}{RGB}{0,100,181}
\definecolor{edgeViolet}{RGB}{99,0,165}

\begin{document}

\title[Tropicalized quartics and tropical curves of genus $3$]{Tropicalized quartics and canonical embeddings for tropical curves of genus $3$}
\author{Marvin Anas Hahn}
\address{Marvin Anas Hahn, Eberhard Karls Universit\"at T\"ubingen, Germany}
\email{marvin-anas.hahn@uni-tuebingen.de}
\author{Hannah Markwig}
\address{Hannah Markwig, Eberhard Karls Universit\"at T\"ubingen, Germany}
\email{hannah@math.uni-tuebingen.de}
\author{Yue Ren}
\address{Yue Ren, Max-Planck-Institut MIS, Leipzig, Germany}
\email{yueren@mis.mpg.de}
\author{Ilya Tyomkin}
\address{Ilya Tyomkin, Ben-Gurion University of the Negev, Israel}
\email{tyomkin@math.bgu.ac.il}

\thanks{The first and second authors were partially supported by the DFG TRR 195 (INST 248/235-1). The fourth author was supported by the Israel Science Foundation (grant No. 821/16). This work was completed during the program ``Tropical Geometry, Amoebas and Polytopes'' at the Institute Mittag-Leffler in spring 2018. The authors would like to thank the institute for its hospitality}


\begin{abstract}
  In \cite{BJMS14}, it was shown that not all abstract non-hyperelliptic tropical curves of genus $3$ can be realized as a tropicalization of a quartic in $\mathbb R^2$.
  In this paper, we focus on the interior of the maximal cones in the moduli space and classify all curves which can be realized as a faithful tropicalization in a tropical plane. Reflecting the algebro-geometric world, we show that these are all curves but the tropicalizations of realizably hyperelliptic algebraic curves.

  Our approach is constructive: For a curve which is not the tropicalization of a hyperelliptic algebraic curve, we explicitly construct a realizable model of the tropical plane in $\mathbb{R}^n$, and a faithfully tropicalized quartic in it. These constructions rely on modifications resp. tropical refinements.
  Conversely, we prove that the tropicalizations of hyperelliptic algebraic curves cannot be embedded in such a fashion. For that, we rely on the theory of tropical divisors and embeddings from linear systems \cite{Amini, HMY}, and recent advances in the realizability of sections of the tropical canonical divisor \cite{MUW}.
\end{abstract}
\maketitle
\vspace{-5mm}  
\setcounter{tocdepth}{1}
\tableofcontents

\vspace{-10mm} 
\section{Introduction}

\subsection{Background and context}
Tropical geometry can be viewed as a piecewise linear shadow of algebraic geometry. In good cases, many geometric properties of an algebraic variety are encoded in its tropicalization as long as the tropicalization is sufficiently ``fine". A problem with this philosophy is that the naive tropicalization depends on the embedding of the algebraic variety, and it is not always clear how to choose the right embedding to work with.

Berkovich analytic geometry offers a way to overcome this problem: Berkovich analytification, which can be viewed as the inverse limit of all tropicalization \cite{Pay09}, does not depend on an embedding and contains complete geometric information about the algebraic variety. The price to pay is that Berkovich analytic spaces are complicated objects, while in tropical geometry we hope to replace an algebraic variety with a simpler object in order to study the algebraic variety by means of combinatorics. A \textit{faithful tropicalization} can be viewed as a hands-on compromise between a naive tropicalization and Berkovich analytification. For curves, it offers a way to avoid infinite graphs by working with a concrete tropicalization, which is combinatorially much easier to digest. On the other hand, it is a tropicalization which captures the wanted geometric properties.

In applications of tropical geometry to algebraic geometry, the question how to capture geometric properties in a tropicalization is foundational, and consequently the question how to construct faithful tropicalizations has attracted lots of attention in the recent years \cite{CHW, CM14, CM17, DraismaPostinghel, KY, Wag17}. In particular, modifications and re-embeddings have already been used to construct faithful tropicalizations \cite{CM14, CM17, LM17}. The tropicalization of a re-embedded variety on a modification is sometimes called a \textit{tropical refinement} \cite{IMS09, MMS09}.

\vspace{-2mm}
\subsection{The question}
It is well-known that any smooth projective non-hyperelliptic curve of genus $3$ can be embedded as a smooth quartic in the plane. Moreover the embedding is given by the complete canonical linear system, and any smooth plane quartic is obtained this way.

In \cite{BJMS14}, Brodsky, Joswig, Morrison and Sturmfels observed that the naive tropical analogue of the above statement does not hold. Namely, not any non-hyperelliptic curve of genus $3$ admits an embedding as a tropical quartic in $\mathbb{R}^2$. More precisely, they computed the locus of embeddable tropical curves, which is far from being the union of all open top-dimensional cones in $M_3^{\trop}\setminus M_{\hyp}^{\trop}$ (see \cite[Theorem~5.1]{BJMS14}). 
With this computational project --- which offers interesting perspectives on computations with secondary fans and tropical moduli spaces beyond this result --- they thus uncovered a discrepancy between the algebraic and tropical world, which as commonly suspected arises due to the use of naive tropicalization of plane curves.

In this paper, we propose a natural setting in which the tropical analogue of the above statement does hold true. Before presenting our setting, let us mention that it is not very difficult to prove (non-constructively) that {\em any} (even hyperelliptic) tropical curve $\Gamma$ of genus 3 can be realized faithfully as a quartic in some tropical model of the plane. To do so, one first proves the existence of a non-hyperelliptic curve $C$ of genus 3 tropicalizing to $\Gamma$ such that $C$ admits an embedding in some $\PP^n$ with faithfull tropicalization (e.g., one can embed $\Gamma$ non-superabundantly in $\RR^3$ and construct $C$ using basic deformation theory, cf. \cite{Tyo12}). Second, one embeds ${\mathcal O}_{\PP^n}(1)|_C$ in $\omega_C^{\otimes m}$ for $m$ large enough to construct a pluricanonical embedding of $C$ with faithful tropicalization. Finally, one observes that the $m$-pluricanonical image of $C$ belongs to the $m$-Veronese embedding of the projective plane. Its tropicalization is the desired model of the plane. However, since abstract tropical models of the plane are very complicated such a tautological tropical analogue is not very satisfactory.

In this paper we consider the class of tropical planes, which are tropicalizations of linear subplanes of projective spaces. To motivate the choice of the class, recall that on the algebraic side, any smooth genus $3$ quartic $C\subset \PP^n$ is embedded by a tuple of canonical sections, hence factors through a planar embedding followed by a linear embedding of $\PP^2$ into $\PP^n$. On the tropical side, it is natural to ask: {\em What tropical curves $\Gamma\in M_3^{\trop}$ admit faithful embeddings in $\RR^n$ as tropical quartics?} Notice that if $\Gamma\hookrightarrow\RR^n$ is such an embedding and $C\hookrightarrow\PP^n$ its faithful realization then $C$ is a planar quartic in $\PP^n$ embedded via a tuple of canonical sections, hence $\Gamma$ sits in the tropicalization of the corresponding plane and the embedding $\Gamma\hookrightarrow\RR^n$ is given by (simultaneously) realizable sections of the tropical canonical divisor.

\subsection{The results}
To formulate our main results we need the following notions: an (embedded) tropical curve is called {\em realizably hyperelliptic} (or {\em rh}) if it is the (embedded) tropicalization of an algebraic hyperelliptic curve and is called {\em {\nrh} curve} otherwise (see Definition \ref{def-hyper}). A tropical curve is called {\em maximal} if it belongs to the open interior of a cone of maximal dimension in the moduli space of tropical curves (see Definition \ref{def-maximal}). 

\begin{theorem}\label{thm-main}
Any maximal {\nrh} tropical curve of genus $3$ can be embedded as a faithfully tropicalized quartic in a linear modification of the tropical plane $\RR^2$ (after attaching unbounded edges appropriately).
\end{theorem}

Our proof of this result is constructive. Namely, for a given maximal {\nrh} curve $\Gamma$ of genus $3$, we explicitly construct a map from $\Gamma$ to $\RR^2$ and a series of linear modifications yielding an embedding of $\Gamma$ into the modified plane as a faithful tropical quartic. Composing with the coordinate projections of $\mathbb{R}^n$, we obtain piecewise linear functions on $\Gamma$. Our construction can be viewed as a \textit{canonical embedding} of $\Gamma$ into $\RR^n$ via an $n$-tuple of canonical divisors. In addition, it provides a way to realize algebraically a curve of genus $3$ together with an $n$-tuple of canonical divisors. An exhaustive set of Examples can be found on \url{https://software.mis.mpg.de}.

\begin{theorem}\label{thm-main2}
No maximal {\em rh} tropical curve $\Gamma$ of genus $3$ can be embedded in $\RR^n$ in such a way that $\Gamma\hookrightarrow\RR^n$ can be realized faithfully by a degree $4$ morphism $C\to\PP^n$, where $C$ is a curve of genus $3$.
\end{theorem}

The proof of this theorem is based on the theory of tropical divisors (cf. \cite{HMY, Amini}), and uses the recent description of the locus of realizable sections of the tropical canonical divisor of M\"oller, Ulirsch, and Werner \cite{MUW}. We shall emphasize that although for some curves there are purely tropical obstructions to the existence of an embedding as a faithfully tropicalized quartic, there also exist hyperelliptic tropical curves that can be embedded in a linearly-modified tropical plane, but which do not come as tropicalizations of plane quartics.

In this paper, we restrict our study to maximal tropical curves. It would be interesting to extend this study to lower-dimensional cones, and our methods are suitable for doing so. We leave this task for further research.

\subsection{Acknowledgements}
We would like to thank Omid Amini, Mar\'ia Ang\'elica Cueto, Dominik Duda, Christian Haase and Michael Joswig for helpful discussions.
Computations have been made using the Computeralgebra system \textsc{Singular} \cite{DGPS}, in particular the library tropical.lib \cite{JMM07a}, and polymake \cite{polymake}, in particular its tropical application \cite{HJ16}. We thank an anonymous referee for helpful comments on an earlier version of this paper.

\section{Preliminaries}
\subsection{The moduli space of tropical curves of genus $3$}\label{subsec:moduliandcurves}

Recall that an \textit{abstract tropical curve} is a connected metric graph $\Gamma$ with possibly a few unbounded half-edges attached (called ends) together with a function associating a genus $g(V)\in \ZZ_{\ge 0}$ to each vertex $V$. A vertex of genus $0$ must be at least $3$-valent.
The genus of $\Gamma$ is defined to be $ g(\Gamma):=b_1(\Gamma) + \sum_{V} g(V)$,
where $b_1(\Gamma)$ denotes the first Betti number.
The \textit{combinatorial type} of a tropical curve is obtained by disregarding the metric structure.

The set of all tropical curves of a given combinatorial type can be parameterized by the quotient of an open orthant in a real vector space under the action of the automorphism group of the combinatorial type. Cones corresponding to different combinatorial types can be glued together by keeping track of possible degenerations under the poset of combinatorial types. In this way, the tropical moduli space $M_g^{\trop}$ of curves of genus $g$ inherits the structure of an abstract cone complex. For more details on tropical moduli spaces of curves, see e.g.\ \cite{ACPModuli, cav.et.at:17, cha:12,GKM,mik:07}.

\begin{definition} \label{def-maximal} We say that an abstract tropical curve is \textit{maximal} if it appears in the interior of a top-dimensional cone of $M_g^{\trop}$.
\end{definition}
In \cite[Figure~1]{cha:12}, the complete poset of  cones in $M_3^{\trop}$ is given. In particular, the maximal tropical curves of genus $3$ belong to one of the following types:
\begin{center}
\tnullnullnull{}
,\qquad \tnulltwonull{}
,\qquad \toneoneone{}
,\qquad \ttwoonetwo{}
,\qquad \tthreenullthree{}
.
\end{center}

Recall that a tropical curve is called {\em hyperelliptic}, if it admits a divisor of degree $2$ and rank $1$, \cite[Definition~2.3]{Chan13}. See also \cite{Baker, HMY} for an introduction to the theory of tropical divisors and linear systems. By \cite[Theorem~3.12]{Chan13}, a tropical curve is hyperelliptic if and only if it either admits a non-degenerate harmonic morphism of degree $2$ to a tree or has exactly two vertices. We shall mention that such a degree $2$ morphism need not be the tropicalization of a degree $2$ morphism from a hyperelliptic curve to the projective line by \cite[Example~4.14]{ABB}.

All top-dimensional cones in $M_3^{\trop}$ but \tnullnullnull{} contain hyperelliptic tropical curves, e.g., in types \tnulltwonull{}, \toneoneone{}, \ttwoonetwo{} the edges forming a $2$-edge cut on a hyperelliptic curve need to have the same lengths. We refer the reader to \cite[Figure~2]{Chan13} for the complete poset of types of hyperelliptic curves of genus $3$.

\begin{definition} \label{def-hyper} We say that a tropical curve is \textit{realizably hyperelliptic} or {\em \rh}, if it is the tropicalization of an algebraic hyperelliptic curve. All other curves are called {\em {\nrh} curves}.
\end{definition}

It follows from \cite[Theorem~4.13]{ABB} that the hyperelliptic tropical curves of types \tnulltwonull{}, \toneoneone{}, \ttwoonetwo{} are \rh, while the curves of type \tthreenullthree{} are not.

\subsection{Tropicalized quartics in $\mathbb{R}^2$}

We assume that the reader is familiar with the basics of tropical geometry, in particular with tropical curves in $\mathbb{R}^2$ and their dual Newton subdivisions. For an introduction to these topics, see e.g.\ \cite{BIMS, Gathmann, FirstSteps}.

For the sake of  explicitness, we pick as ground field the field $K$ of Puiseux series in $t$, with valuation $\valuation$ sending a series to its least exponent \cite{Mar10}.
We use the max-convention, i.e.\ the tropicalization map sends a point $(x_1,\ldots,x_n)\in (K^{\ast})^n$ to $(-\valuation (x_1),\ldots,-\valuation (x_n)) \in \mathbb{Q}^n$.
By $\Trop(X)$ we denote the (naive) tropicalization of an algebraic variety  $X\subset (K^{\ast})^n$.
An example of a picture of a tropicalized quartic can be found in Figure~\ref{fig:type020embedded}.

\subsection{Modifications, re-embeddings and coordinate changes}

\label{sec-modi}

For our purposes, it is sufficient to consider modifications along tropical lines without a vertex, i.e.\ defined by a tropical binomial.
Here, we assume without restriction that we modify at a tropical line defined by the tropical polynomial $L=\max\{0, Y\}$. The graph of $L$
considered as a function on $\RR^2$ consists of two linear
pieces. At the break line, we attach a two-dimensional cell spanned in
addition by the vector $(0,0,-1)$ (see e.g.\ \cite[Construction 3.3]{AR07}).
 We assign multiplicity $1$ to each cell and obtain a
balanced fan in $\RR^3$. It is called the \textit{modification of
  $\RR^2$ along $L$}. If $\Gamma\subset \mathbb{R}^2$ is a plane tropical curve, we bend it analogously and attach downward ends to get the modification of $\Gamma$ along $L$, which now is a tropical curve in the modification of $\RR^2$ along $L$.

Let $\ell=m+y\in K[x,y]$ be a lift of $L$, i.e.\ $-\valuation (m)=0$.
We fix an irreducible polynomial $q\in
K[x,y]$ defining a curve in the torus $(K^{\ast})^2$. The
tropicalization of the variety defined by $I_{q,\ell}=\langle q, z-\ell\rangle\subset K[x,y,z]$
is a tropical curve in the modification of $\RR^2$ along $L$. We call
it the \textit{linear re-embedding} of $\Trop(V(q))$
\textit{with respect to $\ell$}.

We describe $\Trop(V(I_{q,\ell}))$ by means of two projections (see Figure~\ref{fig:tropicalModificationPlusProjections}):
\begin{enumerate}[leftmargin=*]
\item the projection $\pi_{XY}$ to the coordinates $(X,Y)$ produces the
  original tropical curve $\Trop(V(q))$.
\item the projection $\pi_{XZ}$ gives a new tropical plane curve
  $\Trop(V(\tilde{q}))$ inside the projections of the cells $\{Y\geq 0, Z=Y\}$ and $\{Y=0, Z\leq 0\}$, where
  $\tilde{q}=q(x,z-m)$. The polynomial $\tilde{q}$ generates the
elimination ideal $I_{q,\ell}\cap K[x,z]$.
\end{enumerate}

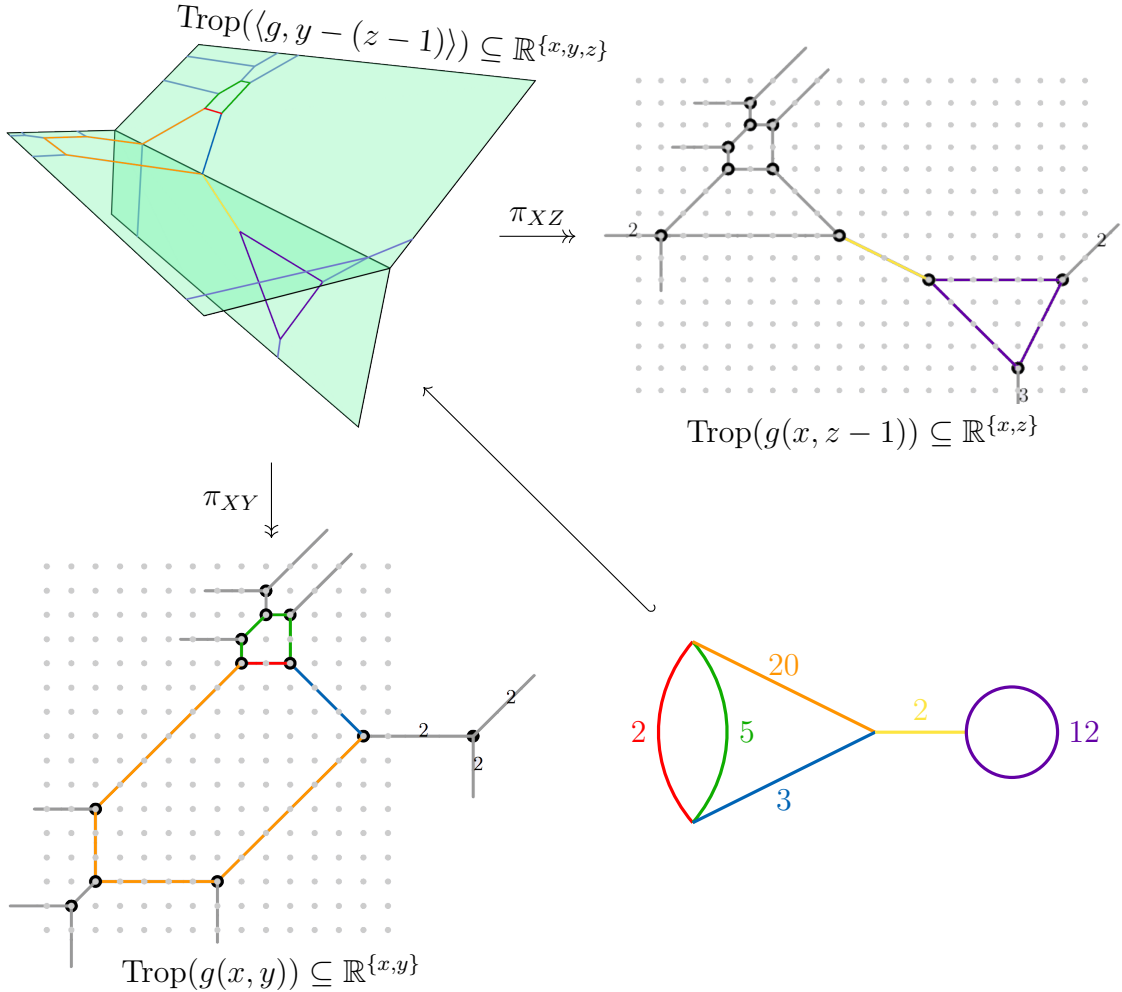
\begin{figure}[h]
  \centering
  \begin{tikzpicture}
    \node (base) at (0,0) {\includegraphics[width=7cm]{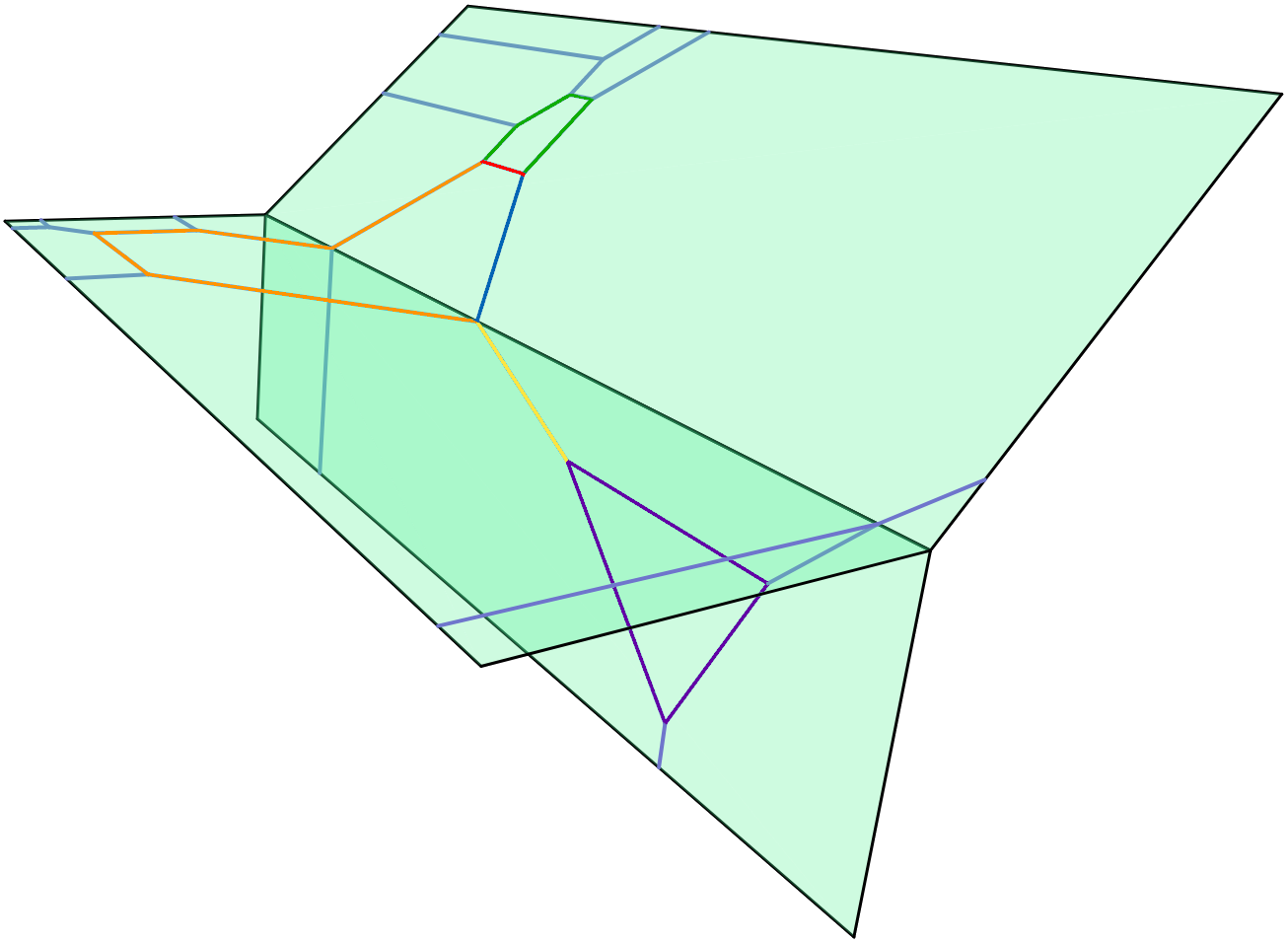}};
    \draw[draw opacity=0] ($(base.north west)+(2.7,-0.2)$) -- node[above,sloped] {$\Trop(\langle g,y-(z-1)\rangle)\subseteq\RR^{\{x,y,z\}}$} ++(5,-0.5);
    \node[anchor=west,xshift=5mm] (topRight) at (base.east) {\includegraphics[width=7cm]{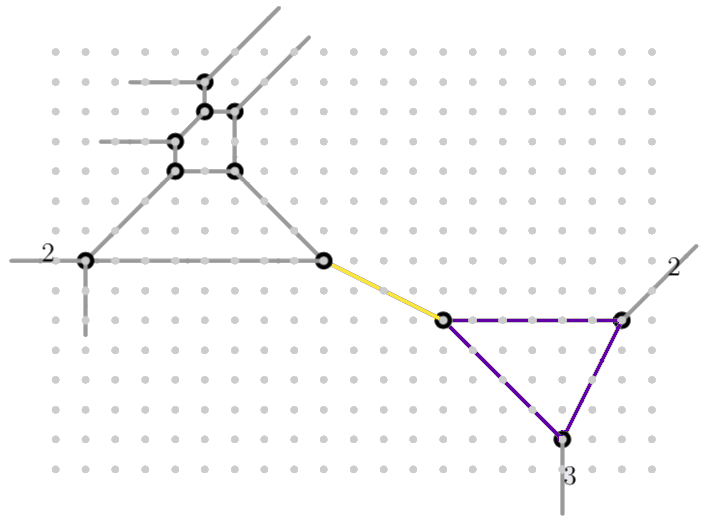}};
    \node[anchor=north,yshift=5mm,fill=white] at (topRight.south) {$\Trop(g(x,z-1))\subseteq\RR^{\{x,z\}}$};
    \node[anchor=north,yshift=-10mm] (bottomLeft) at (base.south) {\includegraphics[width=7cm]{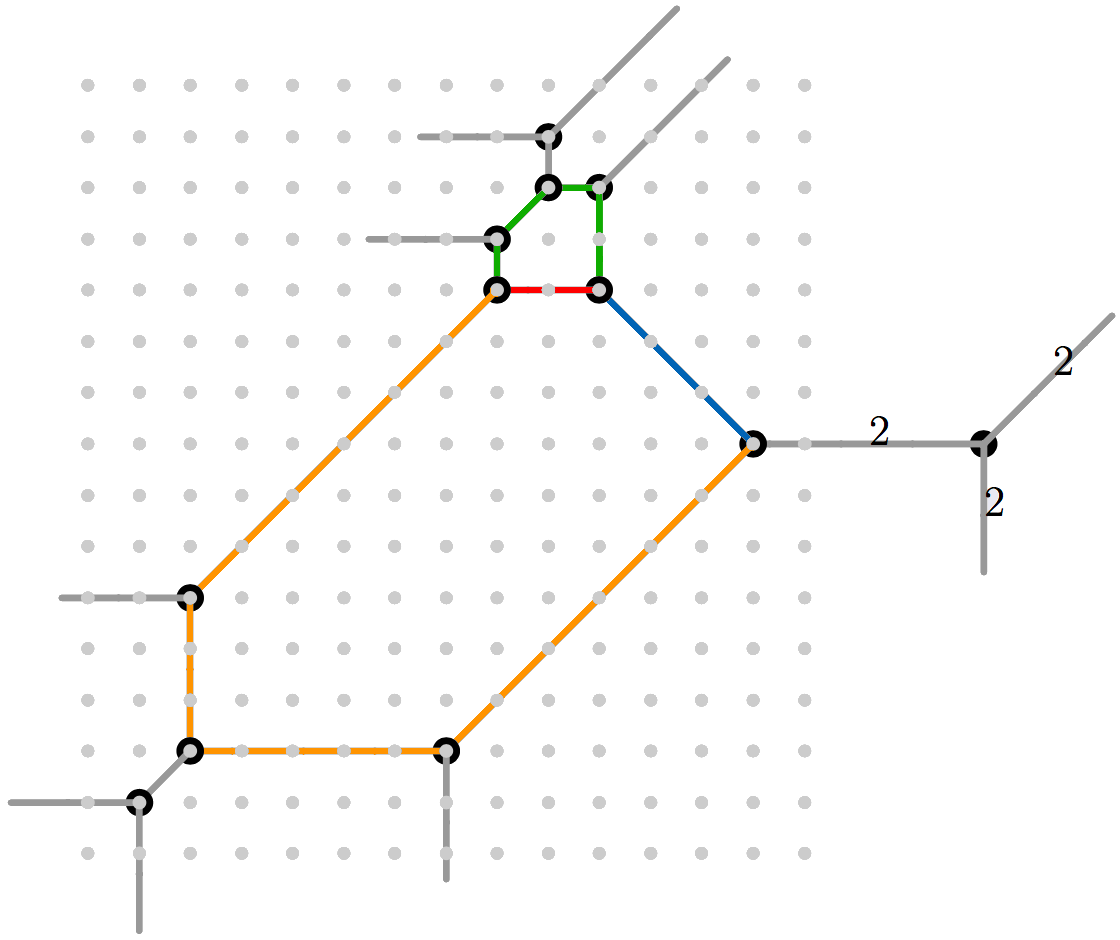}};
    \node[anchor=north,yshift=5mm] at (bottomLeft.south) {$\Trop(g(x,y))\subseteq\RR^{\{x,y\}}$};
    \node[anchor=north west,xshift=8mm,yshift=-25mm] (bottomRight) at (base.south east)
    {
      \begin{tikzpicture}[scale=1.2]
        \draw[edgeRed,very thick] (0,0) to[out=130,in=230] node[anchor=east] {$2$} (0,2);
        \draw[edgeGreen,very thick] (0,0) to[out=50,in=310] node[anchor=west] {$5$} (0,2);
        \draw[edgeOrange,very thick] (0,2) -- node[anchor=south] {$20$} (2,1);
        \draw[edgeBlue,very thick] (0,0) -- node[anchor=north] {$3$} (2,1);
        \draw[edgeYellow,very thick] (2,1) -- node[anchor=south] {$2$} (3,1);
        \draw[edgeViolet,very thick] (3.5,1) circle (0.5);
        \node[edgeViolet,anchor=west] at (4,1) {$12$};
      \end{tikzpicture}
    };
    \draw[left hook->] (5,-5) -- ++(-3,3);
    \draw[->>] (3,0) -- node[above] {$\pi_{XZ}$} ++(1,0);
    \draw[->>] (0,-3) -- node[left] {$\pi_{XY}$} ++(0,-1);
  \end{tikzpicture}\vspace{-2mm}
  \caption{A tropical curve of type \toneoneone{} embeddable into a tropical plane in $\RR^3$ but not into $\RR^2$ itself.}\label{fig:tropicalModificationPlusProjections}
\end{figure}

\begin{lemma}
The linear re-embedding $\Trop(V(I_{q,\ell}))$ in the
  modification of $\RR^2$ along the linear tropical polynomial
  $L=\max\{0, Y\}$ is determined by the two tropical plane
  curves $\Trop(V(q))$ and $\Trop(V(\tilde{q}))$, where
  $\tilde{q}=q(x,z-m)$.
\end{lemma}
For a proof, see \cite[Lemma~2.2]{CM14}.
The content of the proof is to recover the parts of $\Trop(V(I_{q,\ell}))$ which are not contained in the interior of top-dimensional cells --- for the images of the lower codimension cell the preimage under the projection is not unique.
The projection $\Trop(V(\tilde{q}))$ of $\Trop(V(I_{q,\ell}))$
is given by a linear coordinate change of the polynomial $q$.

We can study the Newton subdivision of the projected curve in terms of these coordinate changes:
A term $a\cdot x^i y^j$  of $q$ is replaced by $a\cdot x^i(z-m)^j$, and so it contributes to all terms of the form $x^i z^k$ for $0\leq k\leq j$. This is called the ``feeding process'' and is visualized in \cite[Figure~2]{LM17}.
From the feeding, we can deduce expected valuations of the coefficients. The subdivision corresponding to the expected valuations is dual to the projection of the modified curve, i.e.\ we bend $\Trop(V(q))$ so
that it fits on the graph of $L$, attach downward ends, and project with $\pi_{XZ}$.

We must have cancellation in the coefficients of $\tilde{q}$ to obtain a valuation which is different from the expected valuation, and hence to obtain a linearly re-embedded curve  $\Trop(V(I_{q,\ell}))$  which does not equal the modification of $\Trop(V(q))$ but which makes hidden geometric properties visible.
For almost all lifts $\ell$ of $L$, we do not produce cancellation and thus $\Trop(V(I_{q,\ell}))$  is just a modification. Here, we care for the special choices of $\ell$ which do produce cancellation, so that $\Trop(V(I_{q,\ell}))$ reveals new features.

In \cite{CM14}, such special lifts $\ell$ are constructed for a given curve $V(q)$ such that hidden geometric features become visible in the linear re-embedding. More concretely, \cite{CM14} contains an algorithm that takes a cubic polynomial $q$ and produces finitely many linear re-embeddings in such a way that the linearly re-embedded curve reflects the $j$-invariant.

Here, we construct quartic polynomials $q$ and linear re-embeddings revealing new features of the curve on the modified plane which ``do not fit into $\mathbb{R}^2$''. More precisely, given an abstract tropical curve $\Gamma$ which is not embeddable in $\mathbb{R}^2$ following \cite{BJMS14}, we give a quartic $q$ together with linear re-embeddings such that the linearly re-embedded tropical quartic equals $\Gamma$ (after shrinking leaves).

\subsection{Faithful tropicalization and the forgetful map}

Let $C$ be a smooth algebraic curve over a non-archimedean field $K$. In this paragraph we shall think about $C$ as a smooth projective curve $\overline{C}$ with (marked) punctures. Recall that if $C$ is stable then one can associate to it the stable model over the ring of integers of the field, whose central fiber -- called stable reduction -- is a nodal curve with (marked) punctures satisfying natural stability conditions.

In \cite[\S 2.1.1]{Tyo12}, an abstract tropical curve $\Gamma_C$ was associated to $C$ in a canonical way by considering the dual graph of the stable reduction equipped with a natural metric coming from the geometry of the model. In terms of Berkovich spaces \cite{berkovichbook}, the construction in \cite{Tyo12} is equivalent to considering the minimal skeleton of Berkovich analytification $C^{\an}$, i.e. the minimal $\Gamma\subset C^{\an}$ such that $C^{\an}\setminus\Gamma$ is a disjoint union of discs. Since each component of the stable reduction has a well-defined genus, the tropicalization $\Gamma_C$ comes equipped with a genus function on its vertices. Moreover, $$g(\overline{C})=g(\Gamma):=b_1(\Gamma_C)+\sum_{v\in V(\Gamma_C)}g(v).$$

Let us mention that puncturing $C$ further results in attaching trees of vertices of genus $0$ and unbounded ends to the tropicalization $\Gamma_C$. Thus, the non-trivial geometry of the tropicalization $\Gamma_C$ is contained in $\Gamma_{\overline{C}}$ (at least if $\overline{C}$ itself is stable). Associating $\Gamma_{\overline{C}}$ equipped with the genus function to an algebraic curve $C$ defines the canonical tropicalization map $M_{g}(K)\to M_{g}^{\trop}$.

If the algebraic curve $C$ is embedded in an algebraic torus $(K^{\ast})^n$ then the canonical tropicalization $\Gamma_C$ comes with a natural (not necessarily injective) map $h\colon\Gamma_C\to\RR^n$ turning it into a parameterized tropical curve, see \cite[\S 2.2.1]{Tyo12}. Moreover, $h(\Gamma_C)$ is nothing but the naive tropicalization $\Trop(C)$ obtained by applying the valuation map to the points of $C$ coordinatewise.

Although, unlike the naive tropicalization, the canonical tropicalization is purely intrinsic, it is not easy to compute it explicitly. Thus, it is natural to look for embeddings $C\subset (K^{\ast})^n$, whose naive tropicalizations $\Trop(C)\subset \RR^n$ are fine enough. From the discussion above it is clear that $\Trop(C)$ reflects the metric geometry of $\Gamma_C$ if $h$ is injective at least on the minimal skeleton $\Gamma_{\overline{C}}$. In such a case we say that $\Trop(C)$ is a {\em faithful tropicalization}. Notice that if $\Trop(C)$ is a faithful tropicalization then $h$ maps $\Gamma_{\overline{C}}$ isometrically onto its image by the very definition of the metric on $\Trop(C)$. For example, a smooth tropicalized curve in $\mathbb{R}^2$ (i.e. dual to a unimodular triangulation) is faithful. In general, following~\cite[Theorem 5.24]{BPR1}, one can verify faithfulness by ensuring that the tropical multiplicities of all vertices and edges of the tropical curve are one. For vertices, this means that the corresponding initial degeneration has to be irreducible, for edges -- the weight has to be one.

Let us conclude this section by defining the {\em tropical forgetful map} $\ft^{\trop}$. This is the map that associates to a tropical curve $\Gamma$ an element of $M_g^{\trop}$, $g=g(\Gamma)$, by shrinking the unbounded ends, repeatedly shrinking leaves of genus $0$, and disregarding the vertices of genus $0$ and valency $2$. In particular, if $\overline{C}$ is stable then $\ft^{\trop}(\Gamma_C)=
\Gamma_{\overline{C}}$,
  and if $\Trop(C)$ is a faithful tropicalization then $\ft^{\trop}(\Trop(C))=\Gamma_{\overline{C}}$,
 where we treat $\Trop(C)$ with its graph structure. Readers not familiar with Berkovich analytic curves and stable reductions can nevertheless follow our constructions in Section~\ref{sec-faithfulquartics} to prove Theorem~\ref{thm-main}: they can take our construction as a way to produce {\em a} tropicalization on a realizable model of the tropical plane with a prescribed image under the tropical forgetful map $\ft^{\trop}$.

\section{Constructing faithful tropicalized quartics}\label{sec-faithfulquartics}

This section is dedicated to a constructive proof of Theorem~\ref{thm-main}. For any maximal non-rh tropical curve $\Gamma$ of genus $3$ we explicitly construct a plane quartic curve $C\subseteq K^{2+n}$ whose tropicalization $\Trop(C)$ is faithful and satisfies $\ft^{\trop}(\Trop(C))=\Gamma$. The proof is a case by case analysis depending on the combinatorial type of $\Gamma$:

\begin{center}
  \begin{tabular}{llllll}
    \tnullnullnull{}:& Proposition~\ref{prop:nullnullnull}, &
                                                              \tnulltwonull{}:& Proposition~\ref{prop:nulltwonull}, &
                                                                                                                      \toneoneone{}:& Proposition~\ref{prop:oneoneone},\\
    \ttwoonetwo{}:& Proposition~\ref{prop:twoonetwo},&
                                                       \tthreenullthree{}:& Proposition~\ref{prop:threenullthree}.
  \end{tabular}
\end{center}

We describe $C\subseteq K^{2+n}$ in terms of an ideal $$I:=\langle q, l_1, l_2, \dots, l_n \rangle\subseteq K[x,y,z_1,\ldots,z_n],$$ where $q$ is a quartic polynomial in $x,y$ and $l_1,\dots,l_k$ are linear forms representing a sequence of modifications as described in Section~\ref{sec-modi}.

This section builds on the work of Brodsky-Joswig-Morrison-Sturmfels \cite{BJMS14}, who characterized all tropical curves of genus $3$ that are embeddable into $\RR^2$.

\subsection{Modification techniques}\label{subsec-techniques}
First, we present general constructions that are used multiple times.

\subsubsection{Enlarging edges from trapezoids}\label{subsec:enlarge}
Let $q \in K[x,y]$ be a quartic polynomial whose Newton subdivision contains a trapezoid with a simplex adjacent to its shorter edge as in Figure~\ref{fig:threecases}. For each of the five types, we show how to enlarge the edge $E$ in $\Trop(V(q))$ dual to the shorter edge of the trapezoid. More precisely, assume the edge $E$ of length $f'$ connects two vertices $V$ and $W$ in $\Trop(V(q))$. We construct a suitable linear reembedding given by either $I_{q,x+t^{-l}}$ or $I_{q,y+t^{-l}}$ or $I_{q,y+t^{-l}x}$ with two vertices $V'$ and $W'$ in the $1$-dimensional cell of the modified plane projecting to $V$ and $W$ under $\pi_{XY}$, such that the path connecting $V'$ and $W'$ in the newly attached cell of the modified plane has length $f>f'$.


\begin{figure}[ht]
  \centering
  \begin{minipage}{0.195\linewidth}
    \centering
    \begin{tikzpicture}[scale=0.6]
      \node[anchor=north,font=\footnotesize,fill=white,yshift=-1mm] at (1,0) {(1)};
      \draw (0,0) -- (2,0);
      \draw (0,0) -- (1,1);
      \draw (1,1) -- (1,2);
      \draw[line width=1mm,edgeRed] (1,1) -- (2,1);
      \draw (2,0) -- (2,1);
      \draw (1,2) -- (2,1);
      \fill (0,0) circle (2.5pt);
      \fill (1,0) circle (2.5pt);
      \fill (2,0) circle (2.5pt);
      \fill (1,1) circle (3.5pt);
      \node[fill=white,left,xshift=-1.5mm,font=\tiny,inner sep=0pt,outer sep=0pt] at (1,1) {$(i,j)$};
      \fill (2,1) circle (2.5pt);
      \fill (1,2) circle (2.5pt);
    \end{tikzpicture}
  \end{minipage}%
  \begin{minipage}{0.195\linewidth}
    \centering
    \begin{tikzpicture}[scale=0.6]
      \node[anchor=north,font=\footnotesize,fill=white,yshift=-1mm] at (1,0) {(2)};
      \draw (0,0) -- (2,0);
      \draw (0,0) -- (1,1);
      \draw (1,1) -- (2,2);
      \draw[line width=1mm,edgeRed] (1,1) -- (2,1);
      \draw (2,0) -- (2,2);
      \fill (0,0) circle (2.5pt);
      \fill (1,0) circle (2.5pt);
      \fill (2,0) circle (2.5pt);
      \fill (1,1) circle (3.5pt);
      \node[fill=white,left,xshift=-1.5mm,font=\tiny,inner sep=0pt,outer sep=0pt] at (1,1) {$(i,j)$};
      \fill (2,1) circle (2.5pt);
      \fill (2,2) circle (2.5pt);
    \end{tikzpicture}
  \end{minipage}%
  \begin{minipage}{0.195\linewidth}
    \centering
    \begin{tikzpicture}[scale=0.6]
      \node[anchor=north,font=\footnotesize,fill=white,yshift=-1mm] at (2,0) {(3)};
      \draw (2,2) -- (4,0);
      \draw (1,1) -- (1,2);
      \draw (1,1) -- (2,1);
      \draw[line width=1mm,edgeGreen] (1,2) -- (2,1);
      \draw (1,2) -- (2,2);
      \draw (2,1) -- (4,0);
      \fill (4,0) circle (2.5pt);
      \fill (1,1) circle (2.5pt);
      \fill (2,1) circle (3.5pt);
      \node[fill=white,below,xshift=-2.5mm,yshift=-1mm,font=\tiny,inner sep=0pt,outer sep=0pt] at (2,1) {$(i,j)$};
      \fill (3,1) circle (2.5pt);
      \fill (1,2) circle (2.5pt);
      \fill (2,2) circle (2.5pt);
    \end{tikzpicture}
  \end{minipage}%
  \begin{minipage}{0.195\linewidth}
    \centering
    \begin{tikzpicture}[scale=0.6]
      \node[anchor=north,font=\footnotesize,fill=white,yshift=-1mm] at (1,0) {(4)};
      \draw (0,0) -- (0,2);
      \draw (0,0) -- (1,1);
      \draw (0,2) -- (1,2);
      \draw[line width=1mm,edgeBlue] (1,1) -- (1,2);
      \draw (1,1) -- (2,1);
      \draw (1,2) -- (2,1);
      \fill (0,0) circle (2.5pt);
      \fill (0,1) circle (2.5pt);
      \fill (1,1) circle (3.5pt);
      \node[fill=white,below,xshift=2.5mm,yshift=-1mm,font=\tiny,inner sep=0pt,outer sep=0pt] at (1,1) {$(i,j)$};
      \fill (2,1) circle (2.5pt);
      \fill (0,2) circle (2.5pt);
      \fill (1,2) circle (2.5pt);
    \end{tikzpicture}
  \end{minipage}%
  \begin{minipage}{0.195\linewidth}
    \centering
    \begin{tikzpicture}[scale=0.6]
      \node[anchor=north,font=\footnotesize,fill=white,yshift=-1mm] at (1,0) {(5)};
      \draw (0,0) -- (0,2);
      \draw (0,0) -- (1,1);
      \draw (0,2) -- (1,2);
      \draw[line width=1mm,edgeBlue] (1,1) -- (1,2);
      \draw (1,1) -- (2,2);
      \draw (1,2) -- (2,2);
      \fill (0,0) circle (2.5pt);
      \fill (0,1) circle (2.5pt);
      \fill (1,1) circle (3.5pt);
      \node[fill=white,below,xshift=2.5mm,yshift=-1mm,font=\tiny,inner sep=0pt,outer sep=0pt] at (1,1) {$(i,j)$};
      \fill (2,2) circle (2.5pt);
      \fill (0,2) circle (2.5pt);
      \fill (1,2) circle (2.5pt);
    \end{tikzpicture}
  \end{minipage}
  \caption{Trapezoids within subdivisions}
  \label{fig:threecases}
\end{figure}
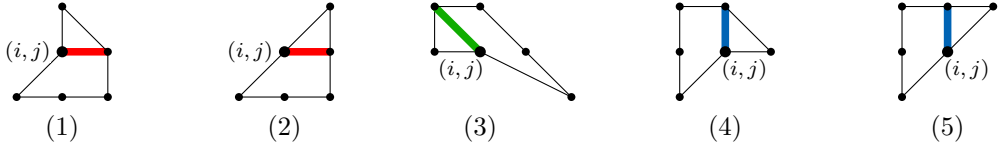

\begin{lemma}\label{lem:unfoldf} Let $q:=\sum a_{mn}x^my^n\in K[x,y]$ be a polynomial whose Newton subdivision contains a trapezoid with a simplex adjacent to its shorter edge as in Figure~\ref{fig:threecases} (1) or (2). Let $f'$ be the length of the edge $E$ of $\Trop(V(q))$ dual to the red edge, and assume $E$ is contained in the line $\{x=l\}$. Let $f>f'$ and suppose the coefficients $a_{mn}$ of $q$ satisfy the following conditions for some $l_1, l_2$ with $l_2<f'$ and $f=2l_1+l_2+f'$:
\begin{align*}
  &\valuation\Big(\sum_{k\geq 0}(-t^{-l})^{k}a_{k,j-1}\Big)\geq f+1,&& \valuation\Big(\sum_{k\geq 0}(-t^{-l})^{k}a_{k,j}\Big)=-l+l_1+l_2,\\
  &\valuation\Big(\sum_{k\geq 0}(-t^{-l})^{k} a_{k,j+1}\Big)=0,&& \valuation\Big(\sum_{k\geq 0}(-t^{-l})^{k-1} \cdot k\cdot  a_{k,j-1}\Big)=l+l_1+f',\\
&\valuation\Big(\sum_{k\geq 0}(-t^{-l})^{k-1}\cdot k\cdot a_{k,j}\Big)=0,&&\valuation\Big(\sum_{k\geq 0}(-t^{-l})^{k-2}\cdot\textstyle\binom{k}{2}\cdot a_{k,j-1}\Big)=f'+2l.
\end{align*}
Then adding the equation $x=z-t^{-l}$ produces an ideal $I$ such that in $\Trop(V(I))$, the edge $E$ is enlarged to length $f$.

\end{lemma}
\begin{proof}
Without restriction, assume that the upper vertex of $E$ is at $(0,0)$, and the valuations of the coefficients corresponding to the triangle dual to $(0,0)$ are $0$ (see Figure~\ref{fig:localf} top). In particular, $l=0$.
We add the equation $x=z-1$ to $q$ and project to the $yz$-coordinates to study the newly attached part of the re-embedded curve $\Trop(V(q,x-z+1))$. The conditions given above are precisely the conditions on the $y^{j-1}$-, $y^j$-, $y^{j+1}$-, $y^{j-1}z$-, $y^jz$- and $y^{j-1}z^2$-terms in $q(z-1,y)$ as imposed in the lower left of Figure~\ref{fig:localf}. The sums above are taken over all terms of $q$ which ``feed'' to the term in question. The valuations determine the positions of the vertices of $\Trop(V(q(z-1,y))$, as in the lower right of Figure~\ref{fig:localf}. By adding the lengths of the edges replacing $E$, it follows that it is enlarged to length $f$.
\end{proof}
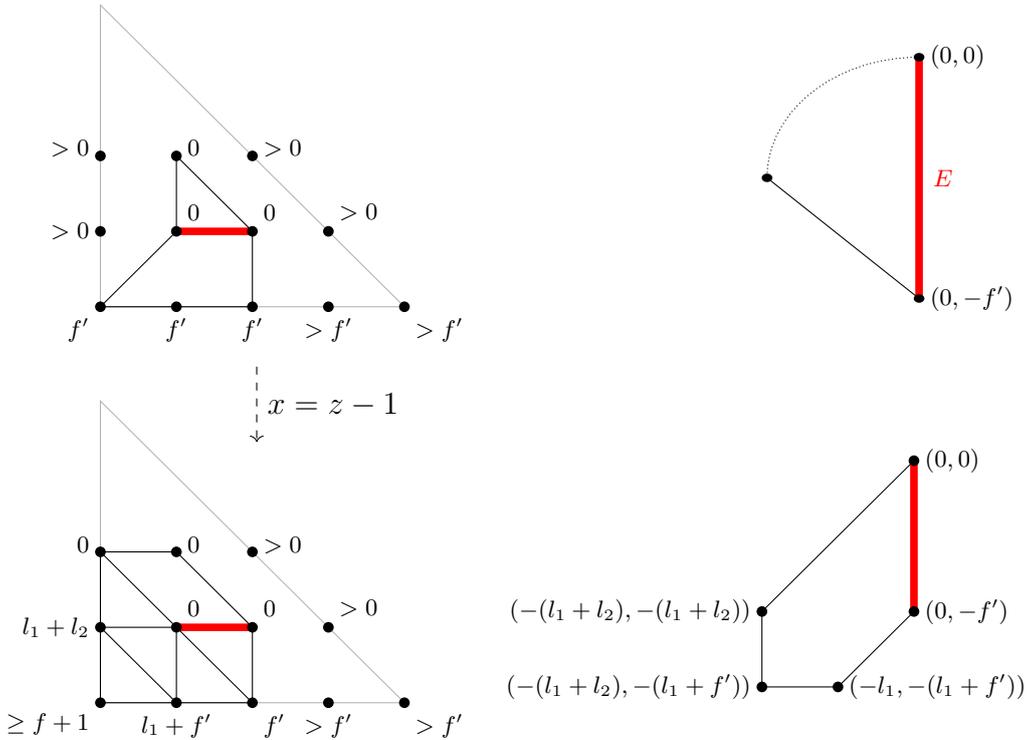
\begin{figure}[h]
  \centering
  \begin{tikzpicture}
    \node[anchor=east] at (0,0)
    {
      \begin{tikzpicture}
        \draw[black!30] (0,0) -- (4,0) -- (0,4) -- (0,0);
        \draw (0,0) -- (2,0);
        \draw (0,0) -- (1,1);
        \draw (1,1) -- (1,2);
        \draw[line width=1mm,edgeRed] (1,1) --  (2,1);
        \draw (2,0) -- (2,1);
        \draw (1,2) -- (2,1);
        \fill (0,0) node[anchor=north east,font=\scriptsize] {$f'$} circle (2pt);
        \fill (1,0) node[anchor=north,font=\scriptsize] {$f'$} circle (2pt);
        \fill (2,0) node[anchor=north,font=\scriptsize] {$f'$} circle (2pt);
        \fill (3,0) node[anchor=north,font=\scriptsize] {$>f'$} circle (2pt);
        \fill (4,0) node[anchor=north west,font=\scriptsize] {$>f'$} circle (2pt);
        \fill (0,1) node[anchor=east,font=\scriptsize] {$>0$} circle (2pt);
        \fill (1,1) node[anchor=south west,font=\scriptsize] {$0$} circle (2pt);
        \fill (2,1) node[anchor=south west,font=\scriptsize] {$0$} circle (2pt);
        \fill (3,1) node[anchor=south west,font=\scriptsize] {$>0$} circle (2pt);
        \fill (0,2) node[anchor=base east,font=\scriptsize] {$>0$} circle (2pt);
        \fill (1,2) node[anchor=base west,font=\scriptsize] {$0$} circle (2pt);
        \fill (2,2) node[anchor=base west,font=\scriptsize] {$>0$} circle (2pt);
      \end{tikzpicture}
    };
    \node[anchor=west] at (3.5,0)
    {
      \begin{tikzpicture}[yscale=0.8,every node/.style={font=\scriptsize}]
        \draw[line width=1mm,edgeRed] (0,0) -- node[right]{$E$} (0,-4);
        \draw (0,-4) -- (-2,-2);
        \draw[densely dotted] (0,0) arc(90:180:2cm);
        \fill (0,0) node[anchor=west] {$(0,0)$} circle (2pt);
        \fill (0,-4) node[anchor=west] {$(0,-f')$} circle (2pt);
        \fill (-2,-2) circle (2pt);
      \end{tikzpicture}
    };
    \node[anchor=east] at (0,-5.25)
    {
      \begin{tikzpicture}[every node/.style={font=\scriptsize}]
        \draw[black!30] (0,0) -- (4,0) -- (0,4) -- (0,0);
        \draw (0,0) -- (0,2);
        \draw (0,2) -- (1,2);
        \draw (0,2) -- (1,1);
        \draw (0,0) -- (2,0);
        \draw (1,0) -- (1,1);
        \draw[line width=1mm,edgeRed] (1,1) -- (2,1);
        \draw (1,1) -- (2,0);
        \draw (2,1) -- (2,0);
        \draw (2,1) -- (1,2);
        \draw (1,1) -- (0,1);
        \draw (0,1) -- (1,0);
        \fill (0,0) node[anchor=north east] {$\geq f+1$} circle (2pt);
        \fill (1,0) node[anchor=north] {$l_1+f'$} circle (2pt);
        \fill (2,0) node[anchor=north west] {$f'$} circle (2pt);
        \fill (3,0) node[anchor=north] {$>f'$} circle (2pt);
        \fill (4,0) node[anchor=north west] {$>f'$} circle (2pt);
        \fill (0,1) node[anchor=east] {$l_1+l_2$} circle (2pt);
        \fill (0,2) node[anchor=base east] {$0$} circle (2pt);
        \fill (1,1) node[anchor=south west] {$0$} circle (2pt);
        \fill (2,1) node[anchor=south west]  {$0$} circle (2pt);
        \fill (3,1) node[anchor=south west] {$>0$} circle (2pt);
        \fill (1,2) node[anchor=base west]  {$0$} circle (2pt);
        \fill (2,2) node[anchor=base west] {$>0$} circle (2pt);
      \end{tikzpicture}
    };
    \node[anchor=west] at (0,-5.25)
    {
      \begin{tikzpicture}[every node/.style={font=\scriptsize}]
        \draw[line width=1mm,edgeRed] (0,0) -- (0,-2);
        \draw (0,-2) -- (-1,-3);
        \draw (-1,-3) -- (-2,-3);
        \draw (-2,-3) -- (-2,-2);
        \draw (-2,-2) -- (0,0);
        \fill (0,0) node[anchor=west] {$(0,0)$} circle (2pt);
        \fill (0,-2) node[anchor=west] {$(0,-f')$} circle (2pt);
        \fill (-1,-3) node[anchor=west] {$(-l_1,-(l_1+f'))$} circle (2pt);
        \fill (-2,-3) node[anchor=east] {$(-(l_1+l_2),-(l_1+f'))$} circle (2pt);east
        \fill (-2,-2) node[anchor=east] {$(-(l_1+l_2),-(l_1+l_2))$} circle (2pt);
      \end{tikzpicture}
    };
    \draw[->,dashed] (-3,-2.5) -- node[right] {$x=z-1$} ++(0,-1);
  \end{tikzpicture}\vspace{-7mm}
  \caption{Enlarging the shorter edge of a trapezoid.}
  \label{fig:localf}
\end{figure}
\begin{remark}
Note that the conditions in Lemma~\ref{lem:unfoldf} can only be fulfilled if there is cancellation behaviour, which requires more than one summand of the same valuation. These equations involving the valuations of the coefficients then imply that there is a trapezoid in the dual subdivision, as in Figure \ref{fig:threecases}.
\end{remark}

The remaining cases represented on Figure~\ref{fig:threecases} are obtained from cases (1) and (2) by an automorphism of the lattice, hence follow formally from Lemma~\ref{lem:unfoldf}. Thus, we state them without a proof:

\begin{lemma}\label{lem:unfolde}
  Let $q:=\sum a_{mn}x^my^n\in K[x,y]$ be a quartic polynomial whose Newton subdivision contains a trapezoid with a simplex adjacent to its shorter edge as in Figure~\ref{fig:threecases} (4) or (5).
Let $e'$ be the length of the edge $E$ of $\Trop(V(q))$ dual to the blue edge, and assume $E$ is contained in the line $\{y=l\}$. Let $e>e'$ and suppose the coefficients $a_{mn}$ of $q$ satisfy the following conditions for some $l_1, l_2$ with $l_2<e'$ and $e=2l_1+l_2+e'$:
  \begin{align*}
  & \valuation \Big( \sum_{k\geq 0} (-t^{-l})^k a_{i-1,k}\Big)\geq e+1,&& \valuation\Big( \sum_{k\geq 0} (-t^{-l})^k a_{i,k}\Big)=-l+l_1+l_2,\\
  & \valuation \Big(\sum_{k\geq 0} (-t^{-l})^k a_{i+1,k}\Big)=-2l,&& \valuation \Big( \sum_{k\geq 0} (-t^{-l})^{k-1}\cdot k\cdot a_{i-1,k}\Big)=l+l_1+e',\\
  &\valuation\Big(\sum_{k\geq 0} (-t^{-l})^{k-1} \cdot k \cdot a_{i,k}\Big)=0,&& \valuation\Big( \sum_{k\geq 0} (-t^{-l})^{k-2} \cdot \textstyle\binom{k}{2}\cdot a_{i-1,k}\Big)=e'+2l.
\end{align*}
 Then adding the equation $y=z-t^{-l}$ produces an ideal $I$ such that in $\Trop(V(I))$, the edge $E$ is enlarged to length $e$.

\end{lemma}

\begin{lemma}\label{lem:unfoldd}
  Let $q:=\sum a_{mn}x^my^n\in K[x,y]$ be a polynomial whose Newton subdivision contains a trapezoid with a simplex adjacent to its shorter edge as in Figure~\ref{fig:threecases} (3). Let $d'$ be the length of the edge $E$ of $\Trop(V(q))$ dual to the green edge, and assume $E$ is contained in the line $\{y=x+l\}$. Let $d>d'$ and suppose the coefficients $a_{mn}$ of $q$ satisfy the following conditions for some $l_1, l_2$ with $l_2<d'$ and $d=2l_1+l_2+d'$:
  \begin{align*}
  & \valuation\Big( \sum_{k\geq 0} (-t^{-l})^k a_{i+j-1-k,k}\Big) =-2l,&& \valuation \Big(\sum_{k\geq 0} (-t^{-l})^k a_{i+j-k,k}\Big)=-l+l_1+l_2,\\
  & \valuation\Big(\sum_{k\geq 0}(-t^{-l})^k a_{i+j+1-k,k}\Big)\geq d+1,&& \valuation\Big( \sum_{k\geq 0} (-t^{-l})^{k-2} \textstyle\binom{k}{2} a_{i+j+1-k,k}\Big)=2l+d',\\
  &\valuation \Big( \sum_{k\geq 0}(-t^{-l})^{k-1}\cdot k\cdot a_{i+j-k,k}\Big)=0,&& \valuation \Big(\sum_{k\geq 0} (-t^{-l})^{k-1} k a_{i+j+1-k,k}\Big)=l+l_1+d' .
\end{align*}
  Then adding the equation $y=z-t^{-l}\cdot x$ produces an ideal $I$ such that in $\Trop(V(I))$, the edge $E$ is enlarged to length $d$.

\end{lemma}

\subsubsection{Unfolding lollipops from weight $2$ edges}
Let $q \in K[x,y]$ be a quartic polynomial whose Newton subdivision contains a simplex of width $2$ as in Figure~\ref{fig:unfoldlollis}. We introduce conditions under which the corresponding bounded weight $2$ edges hide a lolli of prescribed edge and cycle length. Since the three cases differ only by an automorphism of the lattice, we consider only case (1).


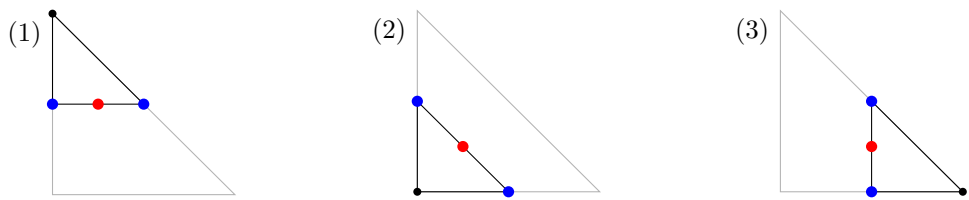
\begin{figure}[ht]
  \centering
  \begin{minipage}{0.32\linewidth}
    \centering
    \begin{tikzpicture}[scale=0.6]
      \draw[black!30] (0,4) -- (0,0) -- (4,0) -- cycle;
      \node[anchor=north east,font=\footnotesize,yshift=0.5mm] at (0,4) {(1)};
      \draw (0,2) -- (2,2) -- (0,4) -- cycle;
      \fill[blue] (0,2) circle (3.5pt);
      \fill[blue] (2,2) circle (3.5pt);
      \fill (0,4) circle (2.5pt);
      \fill[red] (1,2) circle (3.5pt);
    \end{tikzpicture}
  \end{minipage}%
  \begin{minipage}{0.32\linewidth}
    \centering
    \begin{tikzpicture}[scale=0.6]
      \draw[black!30] (0,4) -- (0,0) -- (4,0) -- cycle;
      \node[anchor=north east,font=\footnotesize,yshift=0.5mm] at (0,4) {(2)};
      \draw (0,0) -- (2,0) -- (0,2) -- cycle;
      \fill (0,0) circle (2.5pt);
      \fill[blue] (2,0) circle (3.5pt);
      \fill[blue] (0,2) circle (3.5pt);
      \fill[red] (1,1) circle (3.5pt);
    \end{tikzpicture}
  \end{minipage}%
  \begin{minipage}{0.32\linewidth}
    \centering
    \begin{tikzpicture}[scale=0.6]
      \draw[black!30] (0,4) -- (0,0) -- (4,0) -- cycle;
      \node[anchor=north east,font=\footnotesize,yshift=0.5mm] at (0,4) {(3)};
      \draw (2,0) -- (2,2) -- (4,0) -- cycle;
      \fill[blue] (2,0) circle (3.5pt);
      \fill[blue] (2,2) circle (3.5pt);
      \fill (4,0) circle (2.5pt);
      \fill[red] (2,1) circle (3.5pt);
    \end{tikzpicture}
  \end{minipage}%
  \caption{Subdivisions containing large simplices.}
  \label{fig:unfoldlollis}
\end{figure}

Let $V$ be the vertex dual to the simplex and let $E$ be the weight $2$ edge of $\Trop(V(q))$ dual to $E^\vee:=\conv((0,2),(2,2))$. Denote by $V'$ the other vertex of $E$ (see Figure~\ref{fig:lolli} top).
Without restriction, we can assume that the vertex $V'$ is at $(0,0)$, that the valuations of the coefficients of $y^2$, $xy^2$ and $x^2y^2$ are $0$, and that the coefficient of $x^2y^2$ is $1$.

\begin{lemma}\label{prop-unfoldlollis}
For a tropicalized quartic as above, assume that
\begin{enumerate}[leftmargin=*]
\item \label{prop:unfoldlolliesItem1} $q$ restricted to $E^\vee$ is the square of $y$ times a linear form $L$ in $x$ modulo $\left(t^{2a}\right)$,
\item \label{prop:unfoldlolliesItem2} the $t^{2a}$-contribution of $q|_{E^\vee}$ is not divisible by the $t^{0}$-contribution of $L$,
\item \label{prop:unfoldlolliesItem3} the length of $E$  is $3a+\frac{b}{2} $,
\item \label{prop:unfoldlolliesItem4} the sum of the $t^{0}$-terms of $q$ corresponding to monomials $y\cdot x^i$ for some $i$ is not divisible by the $t^0$-contribution of $L$.
\end{enumerate}
Then we can use one linear modification to unfold a lolli with edge length $a$ and cycle length $b$ from the edge $E$ (see Figure~\ref{fig:lolli} bottom right). 
\end{lemma}
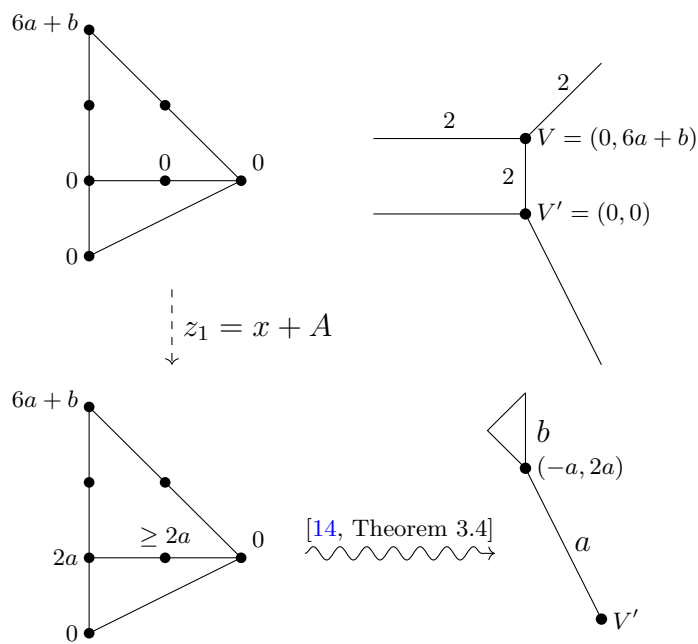
\begin{figure}[ht]
  \begin{center}
    \begin{tikzpicture}[decoration=snake]
      \node[anchor=east] at (0,0)
      {
        \begin{tikzpicture}
          \draw  (2,2)-- (0,4) -- (0,1);
          \draw (0,1) -- (2,2) -- (0,2);
          \fill (0,4) node [anchor=base east,font=\scriptsize] {$6a+b$} circle (2pt);
          \fill (0,3) circle (2pt);
          \fill (1,3) circle (2pt);
          \fill (0,2) node [anchor=east,font=\scriptsize] {$0$} circle (2pt);
          \fill (1,2) node [anchor=south,font=\scriptsize] {$0$} circle (2pt);
          \fill (2,2) node [anchor=south west,font=\scriptsize] {$0$} circle (2pt);
          \fill (0,1) node [anchor=east,font=\scriptsize] {$0$} circle (2pt);
        \end{tikzpicture}
      };
      \node[anchor=west] at (1,-1)
      {
        \begin{tikzpicture}
          \draw (0,0) -- (-1,2);
          \draw (-1,2) --  (-3,2);
          \draw (-1,2) -- node[left,font=\scriptsize] {$2$} (-1,3);
          \draw (-1,3) --  node[above,font=\scriptsize] {$2$} (-3,3);
          \draw (-1,3) -- node[above,font=\scriptsize] {$2$} (0,4);
          \fill 
          (-1,2) circle (0.075cm) node[right,font=\scriptsize] {$V' = (0,0)$}
          (-1,3) circle (0.075cm) node[right,font=\scriptsize] {$V = (0,6a+b)$} ;
        \end{tikzpicture}
      };
      \draw[->,dashed] (-1.5,-2) -- node[right] {$z_1=x+A$} (-1.5,-3);
      \node[anchor=east] at (0,-5)
      {
        \begin{tikzpicture}
          \draw (2,2)-- (0,4) -- (0,1);
          \draw (0,1) -- (2,2) -- (0,2);
          \fill (0,4) node [anchor=base east,font=\scriptsize] {$6a+b$} circle (2pt);
          \fill (0,3) circle (2pt);
          \fill (1,3) circle (2pt);
          \fill (0,2) node [anchor=east,font=\scriptsize] {$2a$} circle (2pt);
          \fill (1,2) node [anchor=south,font=\scriptsize] {$\geq 2a$} circle (2pt);
          \fill (2,2) node [anchor=south west,font=\scriptsize] {$0$} circle (2pt);
          \fill (0,1) node [anchor=east,font=\scriptsize] {$0$} circle (2pt);
        \end{tikzpicture}
      };
      \draw[->,decorate] (0.25,-5.5) -- node[above,font=\scriptsize] {{\cite[Theorem 3.4]{CM14}}}(2.75,-5.5);
      \node[anchor=west] at (2.5,-5)
      {
        \begin{tikzpicture}
          \draw (0,0) node[anchor=west,font=\scriptsize] {$V'$} -- node[right] {$a$} (-1,2) node[anchor=west,font=\scriptsize] {$(-a,2a)$};
          \draw (-1,2) -- node[right] {$b$} (-1,3);
          \draw (-1,2) -- (-1.5,2.5)-- (-1,3);
          \fill (0,0) circle (0.075cm)
          (-1,2) circle (0.075cm);
        \end{tikzpicture}
      };
    \end{tikzpicture}
  \end{center}\vspace{-0.5cm}
  \caption{Unfolding a lolli from weight 2 edges.}
  \label{fig:lolli}
\end{figure}
\begin{proof}
First, we unfold a bridge edge and a loop in two steps. Then we combine the two steps to one.
By condition (1), $q$ restricted to $E^\vee$ equals
$$(y\cdot (x+A))^2+  t^{2a}\cdot y^2 \cdot ( \alpha x+\beta )+O(t^{2a+1}),$$
where $A=A_0+A_1t+A_2t^2+\ldots+A_{2a-1}t^{2a-1}\in K$ is of valuation $0$, and  $\alpha,\beta \in\mathbb{C}$   satisfy either $\alpha=0, \beta \neq 0$ or $0\neq \frac{\beta}{\alpha}\neq A_0$ by condition (2).
By assumption (3), $V$ is at $(0,6a+b)$, and the valuation of the $y^4$-term is $6a+b$.

We start by adding the equation $z_1=x+A$. As described in Section~\ref{sec-modi}, we can make the new part of $\Trop(V(q, z_1-x-A))$ visible using the projection $\pi_{yz_1}$ defined by $\tilde q:=q(z_1-A,y)$.

Condition (4) ensures that there is no cancellation in $\tilde q$ involving the terms $y\cdot x^i$. Thus, the $y$-term of $\tilde q$ has valuation $0$.
For the $y^2$- and $z_1y^2$-term however, there is cancellation. The respective terms of $\tilde q$ are:
$$(y\cdot z_1)^2+  t^{2a}\cdot y^2 \cdot ( \alpha (z_1-A)+\beta )+O(t^{2a+1}),$$

By condition (2), the $y^2$-term has valuation $2a$, and the $y^2z_1$ term has valuation bigger or equal to $2a$.

The Newton subdivision of newly attached part is depicted below in the lower right of Figure~\ref{fig:lolli}. Notice that the bounded edge of weight $2$ is dual to the edge with initial $y^2z_1^2+t^{2a}\cdot(-A_0\alpha+\beta)\cdot y^2= y^2\cdot (z_1 - t^a \cdot \sqrt{-A_0\alpha+\beta}) (z_1+t^a\cdot \sqrt{-A_0\alpha+\beta})$. As this is not a square, it follows from \cite[Theorem~3.4]{CM14} that we can use the linear modification induced by the equation $z_2= z_1 - t^a\cdot \sqrt{-A_0\alpha+\beta}$ to unfold a cycle of length $2\cdot \frac{b}{2}=b$. This cycle is attached to the vertex $(0,0)$ via an edge of length $a$.

The variable $z_1$ is not needed to produce a faithful tropicalization, we can eliminate it and combine the two steps to one.
That is, we only add one equation, namely $z_2= x+A-t^a\cdot \sqrt{-A_0\alpha+\beta}$.
\end{proof}

\subsection{Proof of Theorem~\ref{thm-main}}\label{subsec-constructiveproof}

In this section we present a constructive proof of Theorem~\ref{thm-main}.
Examples for all constructions can be found on \url{https://software.mis.mpg.de}.

\begin{proposition}\label{prop:nulltwonull}
  Let $\Gamma$ be a maximal {\nrh} tropical curve of type $\tnulltwonull{}$.
  Then there exists a plane quartic curve $C\subseteq K^3$ whose tropicalization $\Trop(C)$ is faithful and satisfies $\ft^{\trop}(\Trop(C))=\Gamma$.
\end{proposition}
\begin{proof}
  Let $a,b,c,d,e,f$ be the edge lengths of a tropical curve $\Gamma$ as in Figure~\ref{fig:type020}. Due to symmetry, we may assume $a\leq b$, $c\leq d$, $e\leq f$ and $a\leq e$.
  Since $\Gamma$ is {\nrh}, we have $c<d$.
  We distinguish the following cases:
  \begin{enumerate}
  \item $c+\max\{a,e\}<d,$
  \item $c+e\geq d \text{ and } c+a< d,$
  \item $c+e\geq d \text{ and } c+a\geq d.$
  \end{enumerate}
  \begin{figure}[h]
    \begin{center}
      \begin{minipage}[c][][c]{0.45\linewidth}
        \centering
        \begin{tikzpicture}[scale=1.1]
          \draw[very thick,edgeRed] (0,0) to[out=120,in=240] node[circle,inner sep=0.2mm,outer sep=0.2mm,fill=white] {$a$} (0,2);
          \draw[very thick,edgeGreen] (0,0) to[out=60,in=300] node[circle,inner sep=0.2mm,outer sep=0.2mm,fill=white] {$b$} (0,2);
          \draw[very thick,edgeOrange] (0,2) -- node[circle,inner sep=0.2mm,outer sep=0.2mm,fill=white] {$c$} (3,2);
          \draw[very thick,edgeBlue] (0,0) -- node[circle,inner sep=0.2mm,outer sep=0.2mm,fill=white] {$d$} (3,0);
          \draw[very thick,edgeYellow] (3,0) to[out=120,in=240] node[circle,inner sep=0.2mm,outer sep=0.2mm,fill=white] {$e$} (3,2);
          \draw[very thick,edgeViolet] (3,0) to[out=60,in=300] node[circle,inner sep=0.2mm,outer sep=0.2mm,fill=white] {$f$} (3,2);
          \fill (0,0) circle (1.8pt) node[left,font=\footnotesize]{$B$}
          (0,2) circle (1.8pt) node[left,font=\footnotesize]{$A$}
          (3,0) circle (1.8pt) node[right,font=\footnotesize]{$D$}
          (3,2) circle (1.8pt) node[right,font=\footnotesize]{$C$};
        \end{tikzpicture}
      \end{minipage}
      \begin{minipage}[c][][c]{0.45\linewidth}
        \centering
        \begin{tikzpicture}[scale=0.55]
          \draw[very thick,edgeRed] (0,0) -- node[circle,inner sep=0.2mm,outer sep=0.2mm,fill=white] {$a$} (2,0);
          \draw[very thick,edgeGreen] (0,0) -- (0,1) -- node[circle,inner sep=0.2mm,outer sep=0.2mm,fill=white] {$b$} (2,1) -- (2,0);
          \draw[very thick,edgeOrange] (2,0) -- node[circle,inner sep=0.2mm,outer sep=0.2mm,fill=white] {$c$} (3,-1);
          \draw[very thick,edgeBlue] (3,-3) -- (2,-4) -- node[circle,inner sep=0.2mm,outer sep=0.2mm,fill=white] {$d$} (0,-4) -- (-1,-3) -- (-1,-1) -- (0,0);
          \draw[very thick,edgeYellow] (3,-1) -- node[circle,inner sep=0.2mm,outer sep=0.2mm,fill=white] {$e$} (3,-3);
          \draw[very thick,edgeViolet] (3,-1) -- (5,-1) -- node[circle,inner sep=0.2mm,outer sep=0.2mm,fill=white] {$f$} (5,-3) -- (3,-3);
        \end{tikzpicture}
      \end{minipage}
    \end{center}\vspace{-5mm}
    \caption{A tropical curve and tropicalized quartic in $\RR^2$ of type \tnulltwonull{}.}
    \label{fig:type020}
  \end{figure}
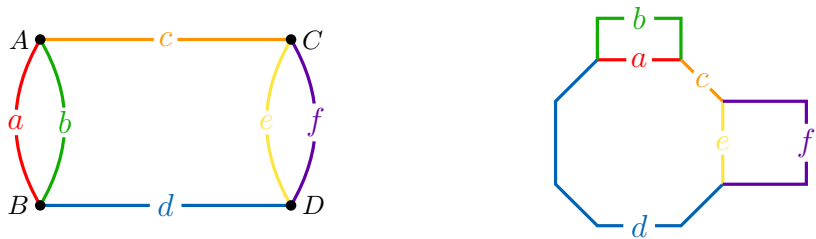
  \begin{enumerate}[leftmargin=*]
  \item By \cite[Theorem 5.1 (020)]{BJMS14}, $\Gamma$ is embeddable as quartic in $\RR^2$.

  \item
    We have to consider two cases: $a=b$ or $a<b$. In both cases, we pick $e'<e$, so that $c+a<c+e'=d$, and $l_1,l_2\geq 0$ such that $2l_1+l_2+e'=e$. If $a=b$, there exist quartic polynomials $q\in K[x,y]$ such that $\Trop(V(q))$ is as depicted in Figure~\ref{fig:NeuesBild1}.

    \begin{figure}[h]
      \begin{center}
        \begin{minipage}{0.45\linewidth}
          \centering
          \begin{tikzpicture}[scale=1.0]
            \fill (0,0) circle (2pt);
            \fill (1,0) circle (2pt);
            \fill (2,0) circle (2pt);
            \fill (3,0) circle (2pt);
            \fill (4,0) circle (2pt);
            \fill (0,1) circle (2pt);
            \fill (1,1) circle (2pt);
            \fill (2,1) circle (2pt);
            \fill (3,1) circle (2pt);
            \fill (0,2) circle (2pt);
            \fill (1,2) circle (2pt);
            \fill (2,2) circle (2pt);
            \fill (0,3) circle (2pt);
            \fill (1,3) circle (2pt);
            \fill (0,4) circle (2pt);
            \node[anchor=north, font=\scriptsize] at (0,0) {$0$};
            \node[anchor=north, font=\scriptsize] at (1,0) {$0$};
            \node[anchor=north, font=\scriptsize] at (2,0) {$0$};
            \node[anchor=north, font=\scriptsize] at (1,1) {$0$};
            \node[anchor=north east, font=\scriptsize] at (2,1) {$0$};
            \draw (0,0) -- (4,0) -- (0,4) -- (0,0)
              (2,0) -- (2,2)
              (0,0) -- (2,2)
              (1,1) -- (0,2)
              (1,1) -- (0,3)
              (1,1) -- (0,4)
              (1,2) -- (0,4)
              (1,2) -- (2,2)
              (2,1) -- (1,1)
              (0,1) -- (1,1)
              (1,3) -- (1,1)
              (3,1) -- (3,0)
              (3,1) -- (2,0)
              (3,1) -- (2,1);
          \end{tikzpicture}
        \end{minipage}
        \begin{minipage}{0.45\linewidth}
          \centering
          \begin{tikzpicture}[xscale=0.2,yscale=0.2]
            \draw (2,1) -- ++(1,1)
            (3,1) -- ++(1,1)
            (8,-1) -- ++(1,1)
            (-3,-1) -- ++(-4,0)
            (-5,-2) -- ++(-2,0)
            (-6,-3) -- ++(-1,0)
            (-6,-4) -- ++(-1,0)
            (8,-14) -- (9,-15) -- (10,-15)
            (10,-15) -- ++(1,1)
            (9,-15) -- ++(0,-1)
            (10,-15) -- ++(0,-1);
            \draw[very thick] (4,-14) -- node[left] {$2$} ++(0,-2);
            \draw[very thick, edgeRed] (0,0) -- node[below,yshift=0.5mm] {$a$} (3,0);
            \draw[very thick, edgeGreen] (0,0) -- node[anchor=south east,xshift=1mm,yshift=-1mm] {$b$} (2,1) -- (3,1) -- (3,0);
            \draw[very thick, edgeOrange] (3,0) -- node[anchor=south west,xshift=-1mm,yshift=-1mm] {$c$} (4,-1);
            \draw[very thick, edgeBlue] (0,0) -- (-3,-1) -- (-5,-2) -- (-6,-3) -- (-6,-4) -- node[circle,inner sep=0.2mm,outer sep=0.2mm,fill=white] {$d$} (4,-14);
            \draw[very thick, edgeYellow] (4,-1) -- node[circle,inner sep=0.2mm,outer sep=0.2mm,fill=white] {$e'$} (4,-14);
            \draw[very thick, edgeViolet] (4,-1) -- (8,-1) -- node[circle,inner sep=0.2mm,outer sep=0.2mm,fill=white] {$f$} (8,-14) -- (4,-14);
          \end{tikzpicture}
        \end{minipage}
      \end{center}\vspace{-3mm}
      \caption{A degenerate tropicalized quartic in $\RR^2$ of type \tnulltwonull{}.}
      \label{fig:NeuesBild1}
    \end{figure}

  If $a<b$, there exist quartic polynomials $q\in K[x,y]$ such that $\Trop(V(q))$ is as depicted in Figure~\ref{fig:type020embedded}.
    \begin{figure}[h]
      \begin{center}
        \begin{minipage}{0.45\linewidth}
          \centering
          \begin{tikzpicture}[scale=1.0]
            \fill (0,0) circle (2pt);
            \fill (1,0) circle (2pt);
            \fill (2,0) circle (2pt);
            \fill (3,0) circle (2pt);
            \fill (4,0) circle (2pt);
            \fill (0,1) circle (2pt);
            \fill (1,1) circle (2pt);
            \fill (2,1) circle (2pt);
            \fill (3,1) circle (2pt);
            \fill (0,2) circle (2pt);
            \fill (1,2) circle (2pt);
            \fill (2,2) circle (2pt);
            \fill (0,3) circle (2pt);
            \fill (1,3) circle (2pt);
            \fill (0,4) circle (2pt);
            \node[anchor=north, font=\scriptsize] at (0,0) {$0$};
            \node[anchor=north, font=\scriptsize] at (1,0) {$0$};
            \node[anchor=north, font=\scriptsize] at (2,0) {$0$};
            \node[anchor=north, font=\scriptsize] at (1,1) {$0$};
            \node[anchor=north east, font=\scriptsize] at (2,1) {$0$};
            \draw (0,0) -- (4,0) -- (0,4) -- (0,0);
            (4,0) -- (0,4);
            \draw (2,0) -- (2,2);
            \draw(0,2) -- (2,2);
            \draw (0,0) -- (2,2);
            \draw (1,1) -- (0,2);
            \draw(2,1)-- (1,1);
            \draw (0,1) -- (1,1);
            \draw (1,3) -- (1,1);
            \draw (1,3) -- (0,3);
            \draw (1,3) -- (0,2);
            \draw (3,1) -- (3,0);
            \draw (3,1) -- (2,0);
            \draw (3,1) -- (2,1);
          \end{tikzpicture}
        \end{minipage}
        \begin{minipage}{0.45\linewidth}
          \centering
          \begin{tikzpicture}[xscale=0.325,yscale=0.325]
            \draw[very thick, edgeOrange] (0,0) -- node[circle,inner sep=0.2mm,outer sep=0.2mm,fill=white] {$c$} (2,-2);
            \draw (0,2) -- (1,3)
            (-2,2) -- (-3,3) -- (-3,4) -- (-2,5)
            (-3,3) -- (-4,3)
            (-3,4) -- (-4,4);
            \draw[very thick, edgeGreen] (0,0) -- (0,2) -- node[circle,inner sep=0.2mm,outer sep=0.2mm,fill=white] {$b$} (-2,2) -- (-2,0);
            \draw[very thick, edgeRed] (-2,0) -- node[circle,inner sep=0.2mm,outer sep=0.2mm,fill=white] {$a$} (0,0);
            \draw[very thick, edgeBlue] (-2,0) -- (-3,-1) -- (-3,-2) -- node[circle,inner sep=0.2mm,outer sep=0.2mm,fill=white] {$d$} (2,-7);
            \draw[very thick, edgeYellow] (2,-7) -- node[circle,inner sep=0.2mm,outer sep=0.2mm,fill=white] {$e'$} (2,-2);
            \draw (-3,-1) -- (-4,-1)
            (-3,-2) -- (-4,-2);
            \draw[very thick] (2,-7) -- node[left,yshift=-2mm] {$2$} (2,-9);
            \draw[very thick, edgeViolet] (2,-7) -- (5,-7) -- node[circle,inner sep=0.2mm,outer sep=0.2mm,fill=white] {$f$} (5,-2) -- (2,-2);
            \draw (5,-2) -- (6,-1)
            (5,-7) -- (6,-8) -- (7,-8) -- (8,-7)
            (6,-8) -- (6,-9)
            (7,-8) -- (7,-9);
            \fill (2,-7) circle (2mm);
            \node[anchor=north east,yshift=2mm,font=\tiny] at (2,-7) {$(0,0)$};
          \end{tikzpicture}
        \end{minipage}
      \end{center}\vspace{-3mm}
      \caption{A degenerate tropicalized quartic in $\RR^2$ of type \tnulltwonull{}.}
      \label{fig:type020embedded}
    \end{figure}

 We take the quartic polynomial $q=\sum_{i,j}a_{ij}x^iy^j$ where
\begin{itemize}[leftmargin=*]
\item all $a_{ij}$ except $a_{11},a_{10},a_{00}$ are of the form $t^{\lambda_{ij}}$ for suitable valuations $\lambda_{ij}$ such that the tropical curve of $q$ is as depicted in Figure~\ref{fig:NeuesBild1} resp.\ \ref{fig:type020embedded}, provided the valuations of $a_{00},a_{10},a_{11}$ are $0$,
\item $a_{11}=1+a_{01}-a_{31}-t^{l_1+l_2}$,
\item $a_{10}=2-3a_{30}+4a_{40}+t^{l_1}$,
\item $a_{00}=1-2a_{30}+3a_{40}+t^{l_1}+t^{e+1}$.
\end{itemize}
Then the conditions of Lemma \ref{lem:unfoldf} are satisfied, i.e.\ adding the equation $x=z-1$ reveals a path of edges of total length $e$ replacing the edge of length $e'$.
\item Pick $e'$ such that $c+e'=d$, $l_1$, $l_2$ such that $2l_1+l_2+e'=e$ and $l_1'$, $l_2'$ such that $2l_1'+l_2'+e'=a$.
There exist quartic polynomials $q=\sum_{i,j}a_{ij}x^iy^j\in K[x,y]$ such that $\Trop(V(q))$ is as depicted in Figure \ref{fig:NeuesBild2}. In particular, we can require that the valuation of $a_{00}$ is $e'$, that $\valuation(a_{22})=0$, $\valuation(a_{13})\geq 2c$, and $a_{11}=1$.

    \begin{figure}[h]
      \begin{center}
        \begin{minipage}{0.45\linewidth}
          \centering
          \begin{tikzpicture}[scale=1.0]
            \fill (0,0) circle (2pt);
            \fill (1,0) circle (2pt);
            \fill (2,0) circle (2pt);
            \fill (3,0) circle (2pt);
            \fill (4,0) circle (2pt);
            \fill (0,1) circle (2pt);
            \fill (1,1) circle (2pt);
            \fill (2,1) circle (2pt);
            \fill (3,1) circle (2pt);
            \fill (0,2) circle (2pt);
            \fill (1,2) circle (2pt);
            \fill (2,2) circle (2pt);
            \fill (0,3) circle (2pt);
            \fill (1,3) circle (2pt);
            \fill (0,4) circle (2pt);
            \node[anchor=south east, font=\scriptsize] at (1,1) {$0$};
            \node[anchor=south west, font=\scriptsize] at (2,1) {$0$};
            \node[anchor=south west, font=\scriptsize] at (2,2) {$0$};
            \draw (0,0) -- (4,0) -- (0,4) -- (0,0)
            (2,0) -- (2,2)
            (0,2) -- (2,2)
            (0,0) -- (2,2)
            (2,1) -- (1,1)
            (1,3) -- (1,1)
            (1,3) -- (0,3)
            (1,3) -- (0,2)
            (3,1) -- (3,0)
            (3,1) -- (2,0)
            (3,1) -- (2,1);
          \end{tikzpicture}
        \end{minipage}
        \begin{minipage}{0.45\linewidth}
          \centering
          \begin{tikzpicture}[scale=0.45]
            \draw (0,1) -- (-1,2) -- (-1,3)
            (5,-3) -- (6,-4) -- (7,-4)
            (-1,3) -- ++(1,1)
            (2,1) -- ++(1,1)
            (5,-1) -- ++(1,1)
            (7,-4) -- ++(1,1)
            (-1,3) -- ++(-1,0)
            (-1,2) -- ++(-1,0)
            (6,-4) -- ++(0,-1)
            (7,-4) -- ++(0,-1);
            \draw[very thick]
            (0,0) -- node[below] {$2$} ++(-2,0)
            (3,-3) -- node[left] {$2$} ++(0,-2);
            \draw[very thick, edgeRed] (0,0) -- node[circle,inner sep=-0.1mm,outer sep=-0.1mm,fill=white] {$e'$} (2,0);
            \draw[very thick, edgeGreen] (0,0) -- (0,1) -- node[circle,inner sep=0.2mm,outer sep=0.2mm,fill=white] {$b$} (2,1) -- (2,0);
            \draw[very thick, edgeOrange] (2,0) -- node[circle,inner sep=0.2mm,outer sep=0.2mm,fill=white] {$c$} (3,-1);
            \draw[very thick, edgeBlue] (0,0) -- node[circle,inner sep=0.2mm,outer sep=0.2mm,fill=white] {$d$} (3,-3);
            \draw[very thick, edgeYellow] (3,-1) -- node[circle,inner sep=-0.1mm,outer sep=-0.1mm,fill=white] {$e'$} (3,-3);
            \draw[very thick, edgeViolet] (3,-1) -- (5,-1) -- node[circle,inner sep=0.2mm,outer sep=0.2mm,fill=white] {$f$} (5,-3) -- (3,-3);
          \end{tikzpicture}
        \end{minipage}
      \end{center}\vspace{-3mm}
      \caption{A degenerate tropicalized quartic in $\RR^2$ of type \tnulltwonull{}.}
      \label{fig:NeuesBild2}
    \end{figure}
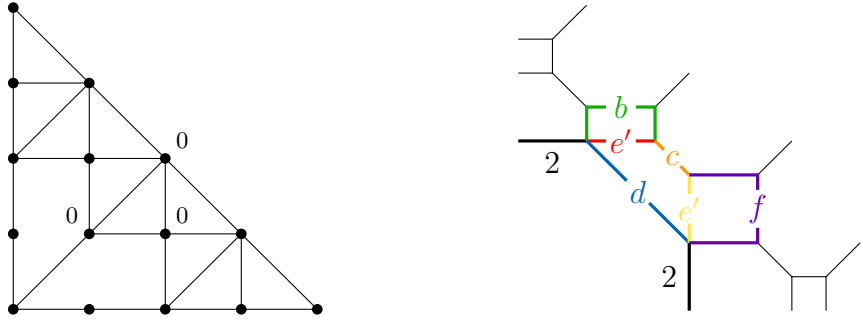

Furthermore, we require that
\begin{itemize}[leftmargin=*]
\item $a_{20}=-t^{l_1+e'}-t^{e+1}+a_{00}+2a_{30}-3a_{40}$
\item $a_{10}=-t^{l_1+e'}+a_{30}-2a_{40}-2t^{e+1}+2a_{00}$
\item $a_{01}=-2t^{c+a+1}+2t^ca_{00}+t^{-2c}a_{03}-2t^{-3c}a_{04}-t^{c+e'+l_1'}$
\item $a_{21}=1-a_{01}+t^{l_1+l_2}$
\item $a_{12}=t^{c+l_1'+l_2'}-t^{2c}a_{10}+t^{-c}a_{13}+t^c$
\item $a_{02}=t^{-2c+a+1}+t^{2c}a_{00}-t^{2c+e'+l_1'}+2t^{-c}a_{03}-3t^{-2c}a_{04}$.
\end{itemize}

Then 
$\valuation(a_{10})=\valuation(a_{20})=e'$, $\valuation(a_{01})=c+e'$, $\valuation(a_{21})=0$, $\valuation(a_{12})=c$, and $\valuation(a_{02})=2c+e'$.
Furthermore, the conditions of Lemmas \ref{lem:unfoldf} and \ref{lem:unfolde} are satisfied, thus adding the equations $z_1=x-1$ and $z_2=y-t^{-c}$ produces paths of length $a$ resp.\ $e$ replacing the edges of length $e'$. \qedhere
  \end{enumerate}
\end{proof}

\begin{proposition}\label{prop:nullnullnull}
  Let $\Gamma$ be a maximal {\nrh} tropical curve of type $\tnullnullnull{}$.
  Then there exists a plane quartic curve $C\subseteq K^4$ whose tropicalization $\Trop(C)$ is faithful and satisfies $\ft^{\trop}(\Trop(C))=\Gamma$.
\end{proposition}
\begin{proof}
  Let $a,b,c,d,e,f$ be the edge lengths of a tropical curve $\Gamma$ as in Figure~\ref{fig:type000}.  Without loss of generality, we may assume that $abc$ is the longest cycle among the four triangular cycles of $\Gamma$. Furthermore, we may assume that $b\leq c\leq a$.

We consider eight possible cases of mutual componentwise inequalities between the triples $(d,e,f)$ and $(c,b,b)$: $$(<<<),(\ge <<),(<\ge <),(<<\ge),(\ge\ge<),(\ge <\ge),(<\ge\ge),(\ge\ge\ge).$$ Notice that by our assumption the case $(\ge\ge\ge)$ is in fact $(= =\ge)$.

  \begin{figure}[h]
    \begin{center}
      \begin{minipage}[c][][c]{0.32\linewidth}
        \centering
        \begin{tikzpicture}[scale=1.5]
          \draw[very thick,edgeGreen] (0,0) -- node[circle,inner sep=0.2mm,outer sep=0.2mm,fill=white] {$b$} (2.32,0);
          \draw[very thick,edgeOrange] (2.32,0) -- node[circle,inner sep=0.2mm,outer sep=0.2mm,fill=white] {$c$} (1.16,2);
          \draw[very thick,edgeRed] (1.16,2) -- node[circle,inner sep=0.2mm,outer sep=0.2mm,fill=white] {$a$} (0,0);
          \draw[very thick,edgeYellow] (0,0) -- node[circle,inner sep=0.2mm,outer sep=0.2mm,fill=white] {$e$} (1.16,0.58);
          \draw[very thick,edgeViolet] (1.16,0.58) -- node[circle,inner sep=0.2mm,outer sep=0.2mm,fill=white] {$f$} (2.32,0);
          \draw[very thick,edgeBlue] (1.16,0.58) -- node[circle,inner sep=0.2mm,outer sep=0.2mm,fill=white] {$d$} (1.16,2);
        \end{tikzpicture}
      \end{minipage}
      \begin{minipage}[c][][c]{0.32\linewidth}
        \centering
        \begin{tikzpicture}[xscale=0.5,yscale=0.4]
          \draw[very thick,edgeRed] (3,3) -- (3,4) -- node[circle,inner sep=0.2mm,outer sep=0.2mm,fill=white] {$a$} (1,3) -- (-1,1) -- (-1,0);
          \draw[very thick,edgeGreen](0,-3) -- (-1,-4) -- (-2,-4) -- node[circle,inner sep=0.2mm,outer sep=0.2mm,fill=white] {$b$} (-3,-3) -- (-3,-2) -- (-1,0);
          \draw[very thick,edgeOrange] (3,3) -- (6,3) -- node[circle,inner sep=0.2mm,outer sep=0.2mm,fill=white] {$c$} (4,-1) -- (1,-3) -- (0,-3);
          \draw[very thick,edgeBlue] (0,0) -- node[circle,inner sep=0.2mm,outer sep=0.2mm,fill=white] {$d$}(3,3);
          \draw[very thick,edgeYellow]	(-1,0) -- node[circle,inner sep=0.2mm,outer sep=0.2mm,fill=white] {$e$} (0,0);
          \draw[very thick,edgeViolet] (0,0) -- node[circle,inner sep=0.2mm,outer sep=0.2mm,fill=white] {$f$} (0,-3);
        \end{tikzpicture}
      \end{minipage}
      \begin{minipage}[c][][c]{0.32\linewidth}
        \centering
        \begin{tikzpicture}[scale=0.8]
          \draw (0,0) -- (4,0);
          \draw (0,0) -- (0,4);
          \draw (0,4) -- (4,0);
          \draw (0,0) -- (1,1);
          \draw (0,1) -- (1,1);
          \draw (0,2) -- (1,2);
          \draw (0,3) -- (1,2);
          \draw (0,4) -- (1,2);
          \draw (1,3) -- (1,2);
          \draw[very thick,edgeYellow] (1,1) -- (1,2);
          \draw (1,0) -- (1,1);
          \draw[very thick,edgeViolet] (1,1) -- (2,1);
          \draw (2,0) -- (2,1);
          \draw[very thick,edgeBlue] (1,2) -- (2,1);
          \draw (1,2) -- (2,2);
          \draw (2,1) -- (3,0);
          \draw (2,1) -- (4,0);
          \draw (2,1) -- (2,2);
          \draw (2,1) -- (3,1);
          \draw (1,1) -- (2,0);
          \draw (1,1) -- (0,2);
          \fill (0,0) circle (2pt);
          \fill (1,0) circle (2pt);
          \fill (2,0) circle (2pt);
          \fill (3,0) circle (2pt);
          \fill (4,0) circle (2pt);
          \fill (0,1) circle (2pt);
          \fill (1,1) circle (2pt);
          \fill (2,1) circle (2pt);
          \fill (3,1) circle (2pt);
          \fill (0,2) circle (2pt);
          \fill (1,2) circle (2pt);
          \fill (2,2) circle (2pt);
          \fill (0,3) circle (2pt);
          \fill (1,3) circle (2pt);
          \fill (0,4) circle (2pt);
        \end{tikzpicture}
      \end{minipage}
    \end{center}\vspace{-5mm}
    \caption{A tropical curve and tropicalized quartic in $\RR^2$ of type \tnullnullnull{}.}
    \label{fig:type000}
  \end{figure}
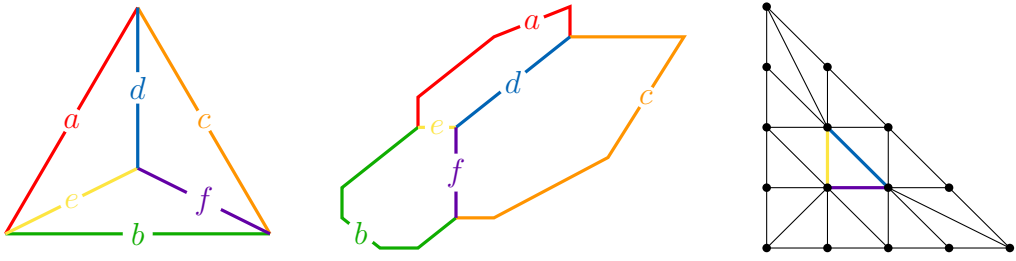

\begin{description}[leftmargin=3mm]
\item[$(<<<)$] By \cite[Theorem 5.1 (000)]{BJMS14}, $\Gamma$ is embeddable as a quartic in $\mathbb{R}^2$.
\item[$(\ge<<)$] Pick a length $d'<c$ and positive $l_1,l_2$ such that $d=d'+2l_1+l_2$. And let us consider a quartic polynomial $q=\sum_{i,j}a_{ij}x^iy^j$ tropicalizing to a tropical curve in $\mathbb{R}^2$ dual to the subdivision on the left of Figure~\ref{fig:dindividually} that realizes the lengths $a,b,c,d',e,f$ and satisfies the following:
\begin{align*}
a_{21}=&a_{12}+a_{30}-a_{03}-t^{l_1+l_2}, \quad a_{31}=2a_{40}-t^{d'+l_1}-2t^{2l_1+l_2+d'+1}+a_{13}-2a_{04},\\
a_{22}=&a_{40}-t^{d'+l_1}-t^{2l_1+l_2+d'+1}+2a_{13}-3a_{04}.
\end{align*}
For example, one can choose $a_{11}=a_{12}=1$, $a_{40}=t^{d'}$, $a_{ij}=t^{\lambda_{ij}}$ for suitable valuations $\lambda_{ij}$ for $(i,j)\notin\{(1,1),(1,2),(2,1),(2,2),(3,1),(4,0)\}$, and use the equations above for $a_{21},a_{31},a_{22}$. Then the coefficients satisfy the equations in Lemma~\ref{lem:unfoldd}, and $\ft^{\trop}(\Trop(V(\langle q,y-z+x\rangle)))=\Gamma.$
\item[$(<\ge<)$] This case follows from the one before by an automorphism of the lattice.
\item[$(<<\ge)$] As before, we could use symmetry argument to deduce this case, but we prefer to give an alternative construction that will be used in the following cases.
Pick a length $f'<b$ and positive $k_1,k_2$ such that $f=f'+2k_1+k_2$. Let $q=\sum_{i,j}a_{ij}x^iy^j$ be a quartic polynomial tropicalizing to a tropical curve in $\mathbb{R}^2$ dual to the subdivision on the right of Figure~\ref{fig:dindividually} that realizes the lengths $a,b,c,d,e,f'$ and satisfies the following:
\begin{align*}
a_{10}=&2a_{00}-t^{k_1+f'}-2t^{2k_1+k_2+f'+1}+a_{30}-2a_{40}, \quad a_{11}=a_{01}+a_{21}-a_{31}-t^{k_1+k_2},\\
a_{20}=&a_{00}-t^{k_1+f'}-t^{2k_1+k_2+f'+1}+2a_{30}-3a_{40}.
\end{align*}
Then $\ft^{\trop}(\Trop(V(\langle q,x-z+1\rangle)))=\Gamma$ by Lemma~\ref{lem:unfoldf}.

   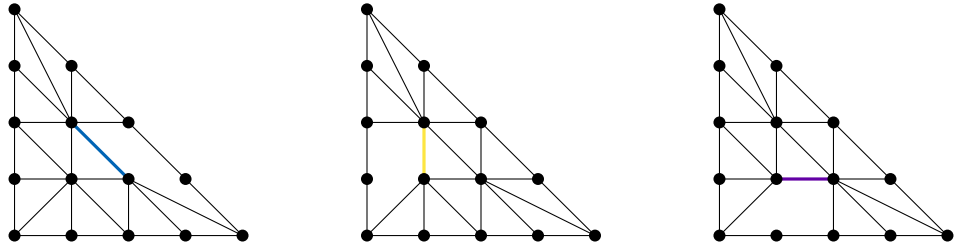
\begin{figure}[ht]
      \begin{center}
        \begin{minipage}{0.3\linewidth}
          \centering
          \begin{tikzpicture}[scale=0.75]
            \draw (0,0) -- (4,0);
            \draw (0,0) -- (0,4);
            \draw (0,4) -- (4,0);
            \draw (0,0) -- (1,1);
            \draw (0,1) -- (1,1);
            \draw (0,2) -- (1,2);
            \draw (0,3) -- (1,2);
            \draw (0,4) -- (1,2);
            \draw (1,3) -- (1,2);
            \draw (1,1) -- (1,2);
            \draw (1,0) -- (1,1);
            \draw (1,1) -- (2,1);
            \draw (2,0) -- (2,1);
            \draw[very thick,edgeBlue] (1,2) -- (2,1);
            \draw (1,2) -- (2,2);
            \draw (2,1) -- (3,0);
            \draw (2,1) -- (4,0);
            \draw (1,1) -- (2,0);
            \draw (1,1) -- (0,2);
            \fill (0,0) circle (3pt);
            \fill (1,0) circle (3pt);
            \fill (2,0) circle (3pt);
            \fill (3,0) circle (3pt);
            \fill (4,0) circle (3pt);
            \fill (0,1) circle (3pt);
            \fill (1,1) circle (3pt);
            \fill (2,1) circle (3pt);
            \fill (3,1) circle (3pt);
            \fill (0,2) circle (3pt);
            \fill (1,2) circle (3pt);
            \fill (2,2) circle (3pt);
            \fill (0,3) circle (3pt);
            \fill (1,3) circle (3pt);
            \fill (0,4) circle (3pt);
          \end{tikzpicture}
        \end{minipage}
        \begin{minipage}{0.3\linewidth}
          \centering
          \begin{tikzpicture}[scale=0.75]
            \draw (0,0) -- (4,0);
            \draw (0,0) -- (0,4);
            \draw (0,4) -- (4,0);
            \draw (0,0) -- (1,1);
            \draw (0,2) -- (1,2);
            \draw (0,3) -- (1,2);
            \draw (0,4) -- (1,2);
            \draw (1,3) -- (1,2);
            \draw[very thick,edgeYellow] (1,1) -- (1,2);
            \draw (1,0) -- (1,1);
            \draw (1,1) -- (2,1);
            \draw (2,0) -- (2,1);
            \draw (1,2) -- (2,1);
            \draw (1,2) -- (2,2);
            \draw (2,1) -- (3,0);
            \draw (2,1) -- (4,0);
            \draw (2,1) -- (2,2);
            \draw (2,1) -- (3,1);
            \draw (1,1) -- (2,0);
            \fill (0,0) circle (3pt);
            \fill (1,0) circle (3pt);
            \fill (2,0) circle (3pt);
            \fill (3,0) circle (3pt);
            \fill (4,0) circle (3pt);
            \fill (0,1) circle (3pt);
            \fill (1,1) circle (3pt);
            \fill (2,1) circle (3pt);
            \fill (3,1) circle (3pt);
            \fill (0,2) circle (3pt);
            \fill (1,2) circle (3pt);
            \fill (2,2) circle (3pt);
            \fill (0,3) circle (3pt);
            \fill (1,3) circle (3pt);
            \fill (0,4) circle (3pt);
          \end{tikzpicture}
        \end{minipage}
        \begin{minipage}{0.3\linewidth}
          \centering
          \begin{tikzpicture}[scale=0.75]
            \draw (0,0) -- (4,0);
            \draw (0,0) -- (0,4);
            \draw (0,4) -- (4,0);
            \draw (0,0) -- (1,1);
            \draw (0,1) -- (1,1);
            \draw (0,2) -- (1,2);
            \draw (0,3) -- (1,2);
            \draw (0,4) -- (1,2);
            \draw (1,3) -- (1,2);
            \draw (1,1) -- (1,2);
            \draw[very thick,edgeViolet] (1,1) -- (2,1);
            \draw (2,0) -- (2,1);
            \draw (1,2) -- (2,1);
            \draw (1,2) -- (2,2);
            \draw (2,1) -- (3,0);
            \draw (2,1) -- (4,0);
            \draw (2,1) -- (2,2);
            \draw (2,1) -- (3,1);
            \draw (1,1) -- (0,2);
            \fill (0,0) circle (3pt);
            \fill (1,0) circle (3pt);
            \fill (2,0) circle (3pt);
            \fill (3,0) circle (3pt);
            \fill (4,0) circle (3pt);
            \fill (0,1) circle (3pt);
            \fill (1,1) circle (3pt);
            \fill (2,1) circle (3pt);
            \fill (3,1) circle (3pt);
            \fill (0,2) circle (3pt);
            \fill (1,2) circle (3pt);
            \fill (2,2) circle (3pt);
            \fill (0,3) circle (3pt);
            \fill (1,3) circle (3pt);
            \fill (0,4) circle (3pt);
          \end{tikzpicture}
        \end{minipage}
      \end{center}\vspace{-0.25cm}
      \caption{Subdivisions of type \tnullnullnull{} for cases $(\ge < <), (< \ge <), (< < \ge)$.}
      \label{fig:dindividually}
    \end{figure}

\item[$(\ge<\ge)$]
%
Choose two lengths $d'<c,f'<b$ and positive $l_1,l_2,k_1,k_2$ such that $d=2l_1+l_2+d'$ and $f=2k_1+k_2+f'$. Let us pick a quartic polynomial $q=\sum_{i,j}a_{ij}x^iy^j$ realizing a curve with lengths $a,b,c,d',e,f'$ dual to a subdivision that contains two trapezoids, as on the left and on the right of Figure~\ref{fig:dindividually}. We can choose $q$ in such a way that $a_{21},a_{31},a_{22}$ and $a_{20},a_{10},a_{11}$ satisfy the equations described in the cases $(\ge <<)$ and $(<<\ge)$ above. Then the modifications can be performed independently, and we obtain
$\ft^{\trop}(\Trop(V(\langle q,y-z'+x,x-z''+1\rangle)))=\Gamma.$
\item[$(<\ge\ge)$, $(\ge\ge<)$] As in the case $(\ge<\ge)$, we independently combine the two constructions for the cases containing exactly one sign $\ge$.

\item[$(\ge\ge\ge)$] 
We first produce a tropical curve with lengths $a,b,c,d=c,e=b,f'$, where $f'=b$ and then enlarge $f'$ to $f$. We pick $l_1$ and $l_2$ such that $f=2l_1+l_2+b$.\par
Consider the quartic polynomial $q=\sum_{i+j\le 4}a_{ij}x^iy^j$, where
\begin{itemize}
\item $a_{04}=t^a$, $a_{03}=a_{13}=0$, $a_{02}=a_{01}=a_{00}t^b$, $a_{30}=2t^b$, $a_{12}=a_{21}=1$, $a_{22}=a_{31}=a_{40}=t^c$,
\item $a_{10}=2a_{00}-t^{l_1+b}-2t^{2l_1+l_2+b+1}+a_{30}-2a_{40}=4t^b-t^{l_1+b}-2t^{2l_1+l_2+b+1}-2t^c$,
\item $a_{11}=a_{01}+a_{21}-a_{31}-t^{l_1+l_2}=1+t^b-t^c-t^{l_1+l_2}$,
\item $a_{20}=a_{00}-t^{l_1+b}-t^{2l_1+l_2+b+1}+2a_{30}-3a_{40}=5t^b-t^{l_1+b}-t^{2l_1+l_2+b+1}-3t^c$.
\end{itemize}

This yields the subdivision in Figure~\ref{fig:gegege}, which produces a curve with lengths $a,b,c,d',e',f'$, where $d'=c,e'=b,f'=b$. Furthermore, the coefficients satisfy the equations in Lemma~\ref{lem:unfoldf}. Analogously to Lemma~\ref{lem:unfoldf}, the re-embedding induced by $z=x+1$ unfolds an edge of length $f$ from the edge of length $f'=b$. \qedhere

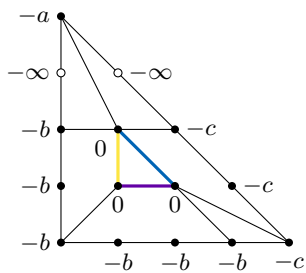
\begin{figure}[ht]
            \begin{tikzpicture}[scale=0.75]
              \draw (0,0) -- (4,0);
              \draw (0,0) -- (0,4);
              \draw (0,4) -- (4,0);
              \draw (0,0) -- (1,1);
              \draw (0,2) -- (1,2);
              \draw (0,4) -- (1,2);
              \draw[very thick,edgeYellow] (1,1) -- (1,2);
              \draw[very thick,edgeViolet] (1,1) -- (2,1);
              \draw[very thick,edgeBlue] (1,2) -- (2,1);
              \draw (1,2) -- (2,2);
              \draw (2,1) -- (3,0);
              \draw (2,1) -- (4,0);
              \fill (0,0) node[anchor=east,font=\scriptsize] {$-b$} circle (2pt);
              \fill (1,0) node[anchor=north,font=\scriptsize] {$-b$} circle (2pt);
              \fill (2,0) node[anchor=north,font=\scriptsize] {$-b$} circle (2pt);
              \fill (3,0) node[anchor=north,font=\scriptsize] {$-b$} circle (2pt);
              \fill (4,0) node[anchor=north,font=\scriptsize] {$-c$} circle (2pt);
              \fill (0,1) node[anchor=east,font=\scriptsize] {$-b$} circle (2pt);
              \fill (1,1) node[anchor=north,font=\scriptsize] {$0$} circle (2pt);
              \fill (2,1) node[anchor=north,font=\scriptsize] {$0$} circle (2pt);
              \fill (3,1) node[anchor=west,font=\scriptsize] {$-c$} circle (2pt);
              \fill (0,2) node[anchor=east,font=\scriptsize] {$-b$} circle (2pt);
              \fill (1,2) node[anchor=north east,font=\scriptsize] {$0$} circle (2pt);
              \fill (2,2) node[anchor=west,font=\scriptsize] {$-c$} circle (2pt);
              \draw[fill=white, draw=black] (0,3) node[anchor=east,font=\scriptsize] {$-\infty$} circle [radius=2pt];
              \draw[fill=white, draw=black] (1,3) node[anchor=west,font=\scriptsize] {$-\infty$} circle [radius=2pt];
              \fill (0,4) node[anchor=east,font=\scriptsize] {$-a$} circle (2pt);
            \end{tikzpicture}
            \caption{Subdivision for type \tnullnullnull{} for the case $(\ge\ge\ge)$.}
            \label{fig:gegege}
            \end{figure}

\end{description}

\end{proof}

\begin{proposition}\label{prop:threenullthree}
  Let $\Gamma$ be a maximal {\nrh} tropical curve of type $\tthreenullthree{}$. Then there exists a plane quartic curve $C\subseteq K^5$ whose tropicalization $\Trop(C)$ is faithful and satisfies $\ft^{\trop}(\Trop(C))=\Gamma$.
\end{proposition}
\begin{proof}
  Let $a,b,c,d,e,f$ be the edge lengths of an tropical curve $\Gamma$ as in Figure~\ref{fig:type303}. By \cite[Theorem 5.1]{BJMS14}, no curve of type \tthreenullthree{} is embeddable as a quartic in $\mathbb{R}^2$. 

  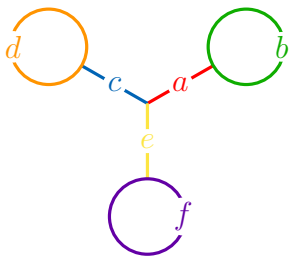
\begin{figure}[h]
    \centering
    \begin{tikzpicture}[scale=0.5]
      \draw[very thick,edgeRed] (0,0) -- node[circle,inner sep=0.2mm,outer sep=0.2mm,fill=white] {$a$} (30:2);
      \draw[very thick,edgeBlue] (0,0) -- node[circle,inner sep=0.2mm,outer sep=0.2mm,fill=white] {$c$} (150:2);
      \draw[very thick,edgeYellow] (0,0) -- node[circle,inner sep=0.2mm,outer sep=0.2mm,fill=white] {$e$} (270:2);

      \node[draw=edgeGreen,very thick,circle,inner sep=0pt,minimum size=1cm] (cycle1) at (30:3) {};
      \node[draw=edgeOrange,very thick,circle,inner sep=0pt,minimum size=1cm] (cycle2) at (150:3) {};
      \node[draw=edgeViolet,very thick,circle,inner sep=0pt,minimum size=1cm] (cycle3) at (270:3) {};

      \node[edgeGreen,circle,inner sep=0.2mm,outer sep=0.2mm,fill=white,xshift=-0.5mm] at (cycle1.east) {$b$};
      \node[edgeOrange,circle,inner sep=0.2mm,outer sep=0.2mm,fill=white,xshift=0.5mm] at (cycle2.west) {$d$};
      \node[edgeViolet,circle,inner sep=0.2mm,outer sep=0.2mm,fill=white,xshift=-0.5mm] at (cycle3.east) {$f$};
    \end{tikzpicture}\vspace{-5mm}
    \caption{An abstract tropical curve of type \tthreenullthree{}.}
    \label{fig:type303}
  \end{figure}

 To show that any curve $C$ is embeddable as a quartic in a tropical plane, let $a,b,c,d,e,f$ be any lengths. Due to symmetry, we may assume $a\ge c\ge e$. Consider the following plane quartic, whose Newton subdivision and tropical curve is depicted in Figure~\ref{fig:303}:
\begin{align*}
  p&= t^{6a+b}\cdot y^4 + t^{6c+d}\cdot x^4+t^{6e+f}  +x^2y^2+ 2\cdot xy^2 + (1-t^{2a}) y^2+ 2\cdot x^2y \\
   &\quad +(1-t^{2c})\cdot x^2- 2\cdot (1+t^{2e})\cdot xy.
\end{align*}

\begin{figure}[h]
  \centering
  \begin{center}
    \begin{minipage}[c][][c]{0.3\linewidth}
      \centering
      \begin{tikzpicture}
        \foreach \x in {0,1,2,3,4}
        {
          \fill (\x,0) circle (0.075cm);
        }
        \foreach \x in {0,1,2,3}
        {
          \fill (\x,1) circle (0.075cm);
        }
        \foreach \x in {0,1,2}
        {
          \fill (\x,2) circle (0.075cm);
        }
        \foreach \x in {0,1}
        {
          \fill (\x,3) circle (0.075cm);
        }
        \foreach \x in {0}
        {
          \fill (\x,4) circle (0.075cm);
        }
        \draw (0,0) -- (4,0) -- (0,4) -- cycle;
        \draw (2,0) -- (2,2) -- (0,2) -- cycle;
      \end{tikzpicture}
    \end{minipage}
    \begin{minipage}[c][][c]{0.55\linewidth}
      \centering
      \begin{tikzpicture}[scale=0.9]
        \draw (0,0) -- node[above] {\scriptsize $2$} (2,0)
        (0,0) -- node[left] {\scriptsize $2$} (0,2)
        (0,0) -- node[below] {\scriptsize $2$} (-1.5,-1.5);

        \draw (2,0) -- node[above] {\scriptsize $2$} ++(1,1)
        (2,0) -- node[right] {\scriptsize $2$} ++(0,-1.5)
        (0,2) -- node[above] {\scriptsize $2$} ++(0.75,0.75)
        (0,2) -- node[above] {\scriptsize $2$} ++(-1.5,0)
        (-1.5,-1.5) -- node[below] {\scriptsize $2$} ++(-1.5,0)
        (-1.5,-1.5) -- node[right] {\scriptsize $2$} ++(0,-0.75);

        \fill (0,0) circle (0.075cm) node[left]{\scriptsize $(0,0)$}
        (2,0) circle (0.075cm) node[right]{\scriptsize $(6c+d,0)$}
        (0,2) circle (0.075cm) node[right]{\scriptsize $(0,6a+b)$}
        (-1.5,-1.5) circle (0.075cm) node[anchor=south east]{\scriptsize $(-6e-f,-6e-f)$};
      \end{tikzpicture}
    \end{minipage}
  \end{center}\vspace{-3mm}
  \caption{The Newton subdivision and plane tropical quartic from which we can unfold three lollis.}
  \label{fig:303}
\end{figure}
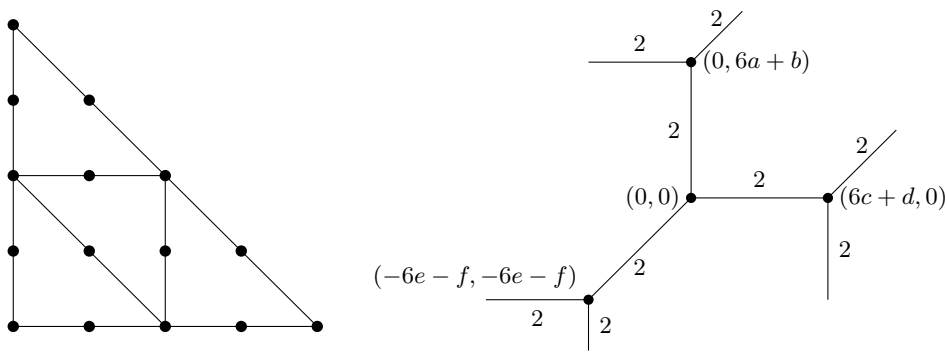
We use Lemma~\ref{prop-unfoldlollis} to unfold three lollis from the edges of weight $2$.
For the polygon $\conv\{(0,2),(2,2),(0,4)\}$, we can read off immediately that the requirements of Lemma~\ref{prop-unfoldlollis} are satisfied:
$p$ restricted to the corresponding weight $2$ edge equals
$$ x^2y^2+2xy^2+(1-t^{2a})\cdot y^2= (y (x+1))^2 +y^2\cdot t^{2a}\cdot (-1).$$
This is a square of $y$ times the linear form $L=x+1$ up to order $t^{2a}$, and the $t^{2a}$-contribution, $-y^2$, is not divisible by $x+1$. Thus conditions (\ref{prop:unfoldlolliesItem1}) and (\ref{prop:unfoldlolliesItem2}) of Lemma~\ref{prop-unfoldlollis} are satisfied. The vertex $V$ is at $(0,6a+b)$ as requested in condition~(\ref{prop:unfoldlolliesItem3}). Finally, the $t^0$-terms of $xy$ and $x^2y$ are $2\cdot (-1+x)xy$ which is again not divisible by $x+1$, so (\ref{prop:unfoldlolliesItem4}) is satisfied. The reasoning for the polygon $\conv\{(2,0),(2,2),(4,0)\}$ is symmetric.
For the polygon $\conv\{(2,0),(0,2),(0,0)\}$, we have to divide by $(1-t^{2a})=\sum_{k\geq0} (t^{2a})^k$ first to express the polynomial in the form used in Lemma~\ref{prop-unfoldlollis}. To simplify the computation, we assume $e\leq c\leq a$.
If $e<c$,
$p$ restricted to the corresponding weight $2$ edge equals
$$ y^2+ x^2- 2\cdot (1+t^{2e})\cdot xy +O(t^{2c})= (x-y)^2 +t^{2e}\cdot x \cdot (-2y) +O(t^{2c}) .$$
If $e=c<a$, we have
$$ y^2+ (1-t^{2e})x^2- 2\cdot (1+t^{2e})\cdot xy +O(t^{2a})= (x-y)^2 +t^{2e}\cdot x \cdot (-2y-x) +O(t^{2a}) .$$
If $e=c=a$, we have
\begin{align*}
  & y^2+ x^2- 2\cdot (1+t^{2e})/(1-t^{2e})\cdot xy +O(t^{2a+1})\\
  & \quad =y^2+ x^2- 2\cdot (1+2t^{2e})\cdot xy +O(t^{2a+1})=
(x-y)^2 +t^{2e}\cdot x \cdot (-4y) +O(t^{2a+1}) .\end{align*}
In any case, the $t^{2e}$-contribution is not divisible by $x-y$.
The $t^0$-contribution of $p$ restricted to $\conv\{(1,2),(2,1)\}$ equals $2\cdot xy\cdot (y+x)$ which is not divisible by $x-y$. With three linear modifications in total, we arrive at a model of the tropical plane and a tropicalized quartic defined by $p$ and the additional linear equations which is faithful on the skeleton and realizes the desired abstract curve.

\end{proof}

\begin{proposition}\label{prop:twoonetwo}
  Let $\Gamma$ be a maximal {\nrh} tropical curve of type $\ttwoonetwo{}$.
  Then there exists a plane quartic curve $C\subseteq K^4$ whose tropicalization $\Trop(C)$ is faithful and satisfies $\ft^{\trop}(\Trop(C))=\Gamma$.
\end{proposition}
\begin{proof}
  Let $a,b,c,d,e,f$ be the edge lengths of a tropical curve $\Gamma$ as in Figure~\ref{fig:type212}. Due to symmetry, we may assume $c\leq d$. By Lemma~\ref{lemtwoonetwo}, $\Gamma$ is rh for $c=d$. We distinguish the following cases:
  \[  (1)\quad c<d\leq 2c \qquad\qquad (2)\quad c<2c<d. \]
  \vspace{-2mm}
  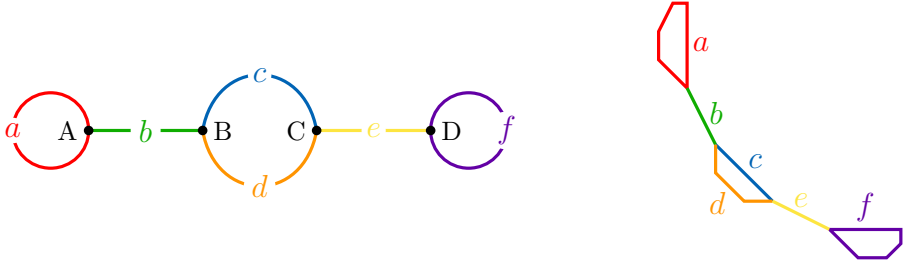
\begin{figure}[ht]
    \begin{center}
      \begin{minipage}[c][][c]{0.45\linewidth}
        \centering
        \begin{tikzpicture}[scale=1]
          \draw[very thick,edgeRed] (0,0) circle (0.5cm);
          \node[edgeRed,circle,inner sep=0.2mm,outer sep=0.2mm,fill=white] at (-0.5,0) {$a$};
          \draw[very thick,edgeGreen] (0.5,0) -- node[circle,inner sep=0.2mm,outer sep=0.2mm,fill=white] {$b$} (2,0);
          \draw[very thick,edgeBlue] (2,0) to[out=80,in=100,min distance=1cm] node[circle,inner sep=0.2mm,outer sep=0.2mm,fill=white] {$c$} (3.5,0);
          \draw[very thick,edgeOrange] (2,0) to[out=280,in=260,min distance=1cm] node[circle,inner sep=0.2mm,outer sep=0.2mm,fill=white] {$d$} (3.5,0);
          \draw[very thick,edgeYellow] (3.5,0) -- node[circle,inner sep=0.2mm,outer sep=0.2mm,fill=white] {$e$} (5,0);
          \draw[very thick,edgeViolet] (5.5,0) circle (0.5cm);
          \node[edgeViolet,circle,inner sep=0.2mm,outer sep=0.2mm,fill=white] at (6,0) {$f$};
          \fill (0.5,0) circle (1.8pt)
          (2,0) circle (1.8pt)
          (3.5,0) circle (1.8pt)
          (5,0) circle (1.8pt);
          \node[left,font=\footnotesize] at (0.5,0) {A};
          \node[right,font=\footnotesize] at (2,0) {B};
          \node[left,font=\footnotesize] at (3.5,0) {C};
          \node[right,font=\footnotesize] at (5,0) {D};
        \end{tikzpicture}
      \end{minipage}
      \begin{minipage}[c][][c]{0.45\linewidth}
        \centering
        \begin{tikzpicture}[scale=0.375]
          \draw[very thick,edgeRed] (-1,2) -- (-2,3) -- (-2,4) -- (-1.5,5) -- (-1,5) -- node[anchor=west,xshift=-2mm,yshift=2.5mm] {$a$} cycle;
          \draw[very thick,edgeGreen] (0,0) -- node[anchor=west,xshift=-0.5mm,yshift=0.5mm] {$b$} (-1,2);
          \draw[very thick,edgeBlue] (0,0) -- node[anchor=south west,xshift=-1mm,yshift=-1mm] {$c$} (2,-2);
          \draw[very thick,edgeOrange] (0,0) -- (0,-1) -- node[anchor=north east,xshift=1mm,yshift=1mm] {$d$} (1,-2) -- (2,-2);
          \draw[very thick,edgeYellow] (2,-2) -- node[anchor=south,yshift=-0.5mm] {$e$} (4,-3);
          \draw[very thick,edgeViolet] (4,-3) -- (5,-4) -- (6,-4) -- (6.5,-3.5) -- (6.5,-3) -- node[above,yshift=-0.5mm] {$f$} (4,-3);
        \end{tikzpicture}
      \end{minipage}
    \end{center}\vspace{-3mm}
    \caption{An abstract and plane tropical curve of type \ttwoonetwo{}.}
    \label{fig:type212}
  \end{figure}

  \begin{enumerate}[leftmargin=*]
  \item By \cite[Theorem~5.1 (212)]{BJMS14}, $\Gamma$ is embeddable as quartic in $\mathbb{R}^2$.
  \item To construct curves satisfying $d>2c$, consider the following quartic polynomial whose Newton subdivision and tropical curve as depicted in Figure \ref{fig:212embedded}:
%
\begin{align*}
  p&= t^{6e+f+2c}\cdot x^4 + t^{a+6b+2c}\cdot y^4  +x^2y^2+ 2\cdot  xy^2 + (1-t^{2b}) y^2+ 2\cdot x^2y \\
   & \quad +(1-t^{2e})\cdot x^2 + t^{-c} xy+ t^{\frac{d-3c-1}{2}}y+ t^{\frac{d-3c-1}{2}} x+ t^{d-3c}.
\end{align*}\vspace{-5mm}

As in the previous type \tthreenullthree{}, Lemma~\ref{prop-unfoldlollis} allows us to independently unfold lollis from the weight $2$ edges, using the modification $z_1=x+1+t^{b}$ and $z_2=y+1+t^e$.
  \end{enumerate}
\end{proof}
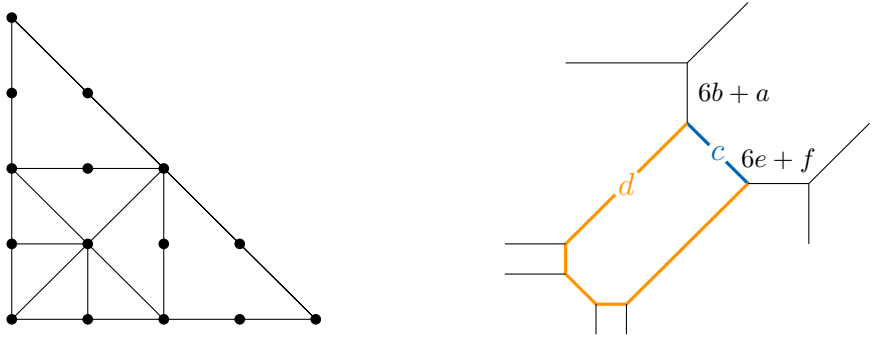
\begin{figure}[ht]
  \begin{center}
    \begin{minipage}{0.4\linewidth}
      \centering
      \begin{tikzpicture}[scale=1.0]
        \fill (0,0) circle (2pt);
        \fill (1,0) circle (2pt);
        \fill (2,0) circle (2pt);
        \fill (3,0) circle (2pt);
        \fill (4,0) circle (2pt);
        \fill (0,1) circle (2pt);
        \fill (1,1) circle (2pt);
        \fill (2,1) circle (2pt);
        \fill (3,1) circle (2pt);
        \fill (0,2) circle (2pt);
        \fill (1,2) circle (2pt);
        \fill (2,2) circle (2pt);
        \fill (0,3) circle (2pt);
        \fill (1,3) circle (2pt);
        \fill (0,4) circle (2pt);
        \draw (0,0) -- (4,0) -- (0,4) -- (0,0)
        (4,0) -- (0,4);
        \draw (2,0) -- (2,2);
        \draw (0,2) -- (2,2);
        \draw (0,0) -- (2,2);
        \draw (2,0) -- (0,2);
        \draw (1,0) -- (1,1);
        \draw (0,1) -- (1,1);
      \end{tikzpicture}
    \end{minipage}
    \begin{minipage}{0.5\linewidth}
      \centering
      \begin{tikzpicture}[scale=0.4]
        \draw[very thick,edgeBlue] (0,0) -- node[circle,inner sep=0.2mm,outer sep=0.2mm,fill=white] {$c$} (2,-2);
        \draw[very thick,edgeOrange] (0,0) -- node[circle,inner sep=0.2mm,outer sep=0.2mm,fill=white] {$d$} (-4,-4) -- (-4,-5) -- (-3,-6) -- (-2,-6)  -- (2,-2);
        \draw (-4,-4) -- (-6,-4);
        \draw (-4,-5) -- (-6,-5);
        \draw (-3,-6) -- (-3,-7);
        \draw (-2,-6) -- (-2,-7);

        \draw (0,0) -- node[right,font=\footnotesize]{$6b+a$} (0,2);
        \draw (2,-2) -- node[above,font=\footnotesize]{$6e+f$}(4,-2);
        \draw (0,2) -- (-4,2);
        \draw (0,2) -- (2,4);
        \draw (4,-2) --(4,-4);
        \draw (4,-2) -- (6,0);
      \end{tikzpicture}
    \end{minipage}
  \end{center}\vspace{-3mm}
  \caption{A tropicalized quartic in $\mathbb{R}^2$ from which we can unfold two lollis.}
  \label{fig:212embedded}
\end{figure}

\begin{proposition}\label{prop:oneoneone}
  Let $\Gamma$ be a maximal {\nrh} tropical curve of type $\toneoneone{}$. Then there exists a plane quartic curve $C\subseteq K^4$ whose tropicalization $\Trop(C)$ is faithful and satisfies $\ft^{\trop}(\Trop(C))=\Gamma$.
\end{proposition}
\begin{proof}
  Let $a,b,c,d,e,f$ be the edge lengths of a tropical curve $\Gamma$ as in Figure~\ref{fig:type111}. Due to symmetry, we may assume $a\geq b$ and $c\leq d$. Since $\Gamma$ is {\nrh}, we have $c<d$. We distinguish the following cases:
  \begin{center}\vspace{-3mm}
    \begin{minipage}[t]{0.37\linewidth}
      \begin{enumerate}[leftmargin=*]
      \item $b+3c\leq d$ and $a>b$,
      \item $b+3c\leq d$ and $a=b$,
      \end{enumerate}
    \end{minipage}%
    \begin{minipage}[t]{0.3\linewidth}
      \begin{enumerate}[leftmargin=*]
        \setcounter{enumi}{2}
      \item $b+c<d<b+3c$,
      \item $c<d\leq b+c$.
      \end{enumerate}
    \end{minipage}%
  \end{center}
  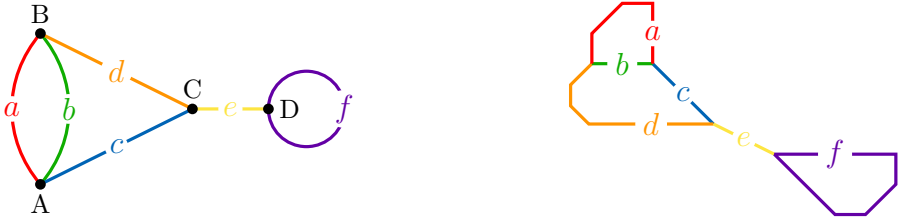
\begin{figure}[ht]
    \begin{center}
      \begin{minipage}[c][][c]{0.4\linewidth}
        \centering
        \begin{tikzpicture}
          \draw[edgeRed,very thick] (0,0) to[out=130,in=230] node[circle,inner sep=0.2mm,outer sep=0.2mm,fill=white] {$a$} (0,2);
          \draw[edgeGreen,very thick] (0,0) to[out=50,in=310] node[circle,inner sep=0.2mm,outer sep=0.2mm,fill=white] {$b$} (0,2);
          \draw[edgeOrange,very thick] (0,2) -- node[circle,inner sep=0.2mm,outer sep=0.2mm,fill=white] {$d$} (2,1);
          \draw[edgeBlue,very thick] (0,0) -- node[circle,inner sep=0.2mm,outer sep=0.2mm,fill=white] {$c$} (2,1);
          \draw[edgeYellow,very thick] (2,1) -- node[circle,inner sep=0.2mm,outer sep=0.2mm,fill=white] {$e$} (3,1);
          \draw[edgeViolet,very thick] (3.5,1) circle (0.5);
          \node[edgeViolet,circle,inner sep=0.2mm,outer sep=0.2mm,fill=white] at (4,1) {$f$};
          \fill (0,0) circle (2pt) node [below,font=\footnotesize] {A};
          \fill (0,2) circle (2pt) node [above,font=\footnotesize] {B};
          \fill (2,1) circle (2pt) node [above,font=\footnotesize] {C};
          \fill (3,1) circle (2pt) node [right,font=\footnotesize] {D};
        \end{tikzpicture}
      \end{minipage}
      \begin{minipage}[c][][c]{0.55\linewidth}
        \centering
        \begin{tikzpicture}[scale=0.4]
          \draw[very thick,edgeBlue] (0,0) -- node[circle,inner sep=0.2mm,outer sep=0.2mm,fill=white] {$c$} (2,-2);
          \draw[very thick,edgeOrange] (2,-2) -- node[circle,inner sep=0.2mm,outer sep=0.2mm,fill=white] {$d$} (-2.1,-2) -- (-2.7,-1.4) -- (-2.7,-0.7) -- (-2,0);
          \draw[very thick,edgeGreen] (-2,0) -- node[circle,inner sep=0.2mm,outer sep=0.2mm,fill=white] {$b$} (0,0);
          \draw[very thick,edgeRed] (0,0) -- node[circle,inner sep=0.2mm,outer sep=0.2mm,fill=white] {$a$} (0,2) -- (-1,2) -- (-2,1) -- (-2,0);
          \draw[very thick,edgeYellow] (2,-2) -- node[circle,inner sep=0.2mm,outer sep=0.2mm,fill=white] {$e$} (4,-3);
          \draw[very thick,edgeViolet] (4,-3) -- (6,-5) -- (7,-5) -- (8,-4) -- (8,-3) --  node[circle,inner sep=0mm,outer sep=0mm,fill=white] {$f$} (4,-3);
        \end{tikzpicture}
      \end{minipage}
    \end{center}\vspace{-4mm}
    \caption{An abstract and plane tropical curve of type \toneoneone{}.}
    \label{fig:type111}
  \end{figure}

  \begin{enumerate}[leftmargin=*]
  \item We assume $b+3c\leq d$ and $a>b$. We consider a plane quartic as in Figure~\ref{fig:111_1_1}. More precisely, we take a quartic polynomial $q=\sum_{i,j}a_{ij}x^iy^j$, where
  \begin{itemize}
  \item all $a_{ij}$ except $a_{20},a_{21}$ are of the form $t^{\lambda_{ij}}$, such that the tropical curve of $q$ is as depicted in Figure~\ref{fig:111_1_1}, provided the valuations of $a_{11},a_{22}$ are $0$ and
  \item $a_{20}=1-t^{2e}$,
  \item $a_{21}=2$.
  \end{itemize}
  Then the conditions of Lemma~\ref{prop-unfoldlollis} are fullfilled and the linear re-embedding via $z=y+1+t^{2e}$ unfolds a lolli of lengths $f$ attached to a stick of length $e$.
  \begin{figure}[h]
    \begin{center}
      \begin{minipage}{0.4\linewidth}
        \centering
        \begin{tikzpicture}[scale=1.0]
          \draw (0,0) -- (4,0) -- (0,4) -- (0,0);
          \draw (1,0) -- (1,3);
          \draw
          (2,0) -- (2,2)
          (2,0) -- (1,1)
          (1,1) -- (2,2)
          (0,1) -- (1,0)
          (0,1) -- (1,1)
          (0,2) -- (1,2)
          (0,2) -- (1,1)
          (0,3) -- (1,3)
          (1,2) -- (2,2)
          (0,3) -- (1,2);
          \fill (0,0) circle (2pt);
          \fill (1,0) circle (2pt);
          \fill (2,0) circle (2pt);
          \draw[fill=white] (3,0) circle (2pt);
          \fill (4,0) circle (2pt);
          \fill (0,1) circle (2pt);
          \fill (1,1) circle (2pt);
          \fill (2,1) circle (2pt);
          \draw[fill=white] (3,1) circle (2pt);
          \fill (0,2) circle (2pt);
          \fill (1,2) circle (2pt);
          \fill (2,2) circle (2pt);
          \fill (0,3) circle (2pt);
          \fill (1,3) circle (2pt);
          \fill (0,4) circle (2pt);
        \end{tikzpicture}
      \end{minipage}
      \begin{minipage}{0.5\linewidth}
        \centering
        \begin{tikzpicture}[scale=0.5]
          \draw (2,-2) -- node[above] {$2$} (4,-2)
          (4,-2) -- node[anchor=base east] {$2$} (5,-1)
          (4,-2) -- node[right] {$2$}(4,-3);
          \draw[very thick,edgeBlue] (0,0) -- node[circle,inner sep=0.2mm,outer sep=0.2mm,fill=white] {$c$} (2,-2);
          \draw[very thick,edgeOrange] (2,-2) -- (1,-3) -- node[circle,inner sep=0.2mm,outer sep=0.2mm,fill=white] {$d$} (-4,-3) -- (-4,-2) -- (-2,0);
          \draw[very thick,edgeGreen,yshift=-0.5mm] (-2,0) -- node[circle,inner sep=0.2mm,outer sep=0.2mm,fill=white] {$b$} (0,0);
          \draw (-4,-3) -- (-5,-4)
          (-5,-4) -- (-6,-4)
          (-5,-4) -- (-5,-5)
          (1,-3) -- (1,-5)
          (-4,-2) -- (-6,-2);
          \draw[very thick,edgeRed] (0,0) -- node[circle,inner sep=0.2mm,outer sep=0.2mm,fill=white] {$a$} (0,2) -- (-1,2) -- (-2,1) -- (-2,0);
          \draw (-2,1) -- (-4,1)
          (0,2) -- (1,3)
          (-1,2) -- (-1,3)
          (-1,3) -- (-2,3)
          (-1,3) -- (0,4);
          \fill (2,-2) circle (3pt);
        \end{tikzpicture}
      \end{minipage}
    \end{center}\vspace{-0.25cm}
    \caption{A curve of type \toneoneone{} with $d\geq b+3c$, $a<b$, and a hidden lolli with edge lengths $e,f$.}
    \label{fig:111_1_1}
  \end{figure}
\item If $d\geq b+3c$ and $a=b$, we can further degenerate the picture above. We consider a plane quartic as in Figure~\ref{fig:111_1_1deg}. More precisely, we take a quartic polynomial $q=\sum_{i,j}a_{ij}x^iy^j$, where
  \begin{itemize}
  \item all $a_{ij}$ except $a_{20},a_{21}$ are of the form $t^{\lambda_{ij}}$, such that the tropical curve of $q$ is as depicted in Figure~\ref{fig:111_1_1}, provided the valuations of $a_{11},a_{22}$ are $0$ and
  \item $a_{20}=1-t^{2e}$,
  \item $a_{21}=2$,
  \end{itemize}
  Then the conditions of Lemma~\ref{prop-unfoldlollis} are fullfilled and the linear re-embedding via $z=y+1+t^{2e}$ unfolds a lolli of lengths $f$ attached to a stick of length $e$. Moreover, using \cite[Theorem~3.4]{CM14}, we can unfold an edge of length $a=b$, which forms a cycle with the edge $b$.
    \begin{figure}[h]
    \begin{center}
      \begin{minipage}{0.4\linewidth}
        \centering
        \begin{tikzpicture}[scale=1.0]
          \draw (0,0) -- (4,0) -- (0,4) -- (0,0);
          \draw (1,0) -- (1,3);
          \draw
          (2,0) -- (2,2)
          (2,0) -- (1,1)
          (1,1) -- (2,2)
          (0,1) -- (1,0)
          (0,1) -- (1,1)
          (0,2) -- (1,1)
          (0,3) -- (1,3);
          \fill (0,0) circle (2pt);
          \fill (1,0) circle (2pt);
          \fill (2,0) circle (2pt);
          \draw[fill=white] (3,0) circle (2pt);
          \fill (4,0) circle (2pt);
          \fill (0,1) circle (2pt);
          \fill (1,1) circle (2pt);
          \fill (2,1) circle (2pt);
          \draw[fill=white] (3,1) circle (2pt);
          \fill (0,2) circle (2pt);
          \fill (1,2) circle (2pt);
          \fill (2,2) circle (2pt);
          \fill (0,3) circle (2pt);
          \fill (1,3) circle (2pt);
          \fill (0,4) circle (2pt);
        \end{tikzpicture}
      \end{minipage}
      \begin{minipage}{0.5\linewidth}
        \centering
        \begin{tikzpicture}[scale=0.5]
          \draw[very thick] (2,-2) -- node[above] {$2$} (4,-2)
          (4,-2) -- node[anchor=base east] {$2$} (5,-1)
          (4,-2) -- node[right] {$2$}(4,-3);
          \draw[very thick,edgeBlue] (0,0) -- node[circle,inner sep=0.2mm,outer sep=0.2mm,fill=white] {$c$} (2,-2);
          \draw[very thick,edgeOrange] (2,-2) -- (1,-3) -- node[circle,inner sep=0.2mm,outer sep=0.2mm,fill=white] {$d$} (-4,-3) -- (-4,-2) -- (-2,0);
          \draw[very thick,edgeGreen] (-2,0) -- node[above] {$b$} node[below] {2} (0,0);
          \draw (-4,-3) -- (-5,-4)
          (-5,-4) -- (-6,-4)
          (-5,-4) -- (-5,-5)
          (1,-3) -- (1,-5)
          (-4,-2) -- (-6,-2);
          \draw (-2,1) -- (-4,1)
          (0,0) -- (1,1)
          (-2,0) -- (-2,1)
          (-2,1) -- (-1,2);
          \fill (2,-2) circle (3pt);
        \end{tikzpicture}
      \end{minipage}
    \end{center}\vspace{-0.25cm}
    \caption{A curve of type \toneoneone{} with $d\geq b+3c$, $a<b$, and a hidden lolli with edge lengths $e,f$.asdf}
    \label{fig:111_1_1deg}
  \end{figure}
\item By \cite[Theorem~5.1 (111)]{BJMS14}, $\Gamma$ is embeddable as a tropicalized quartic in $\RR^2$.
\item
  Now assume $c<d\leq b+c$. Picking $b'$ such that $b'+c=d<b'+3c$, we can consider a tropical plane quartic as in Figure~\ref{fig:111_2_1}.
More precisely, we take the quartic polynomial $q=\sum_{i,j}a_{ij}x^iy^j$ where
\begin{itemize}[leftmargin=*]
\item all $a_{ij}$ except $a_{11},a_{01},a_{00}$ are of the form $t^{\lambda_{ij}}$ for suitable valuations $\lambda_{ij}$ such that the tropical curve of $q$ is as depicted in Figure~\ref{fig:111_2_1}, provided the valuations of $a_{00},a_{10},a_{11}$ are $0$, and
\item $a_{11}=1+a_{10}-a_{13}-t^{l_1+l_2}$,
\item $a_{10}=2-3a_{03}+4a_{04}+t^{l_1}$,
\item $a_{00}=1-2a_{03}+3a_{04}+t^{l_1}+t^{b+1}$.
\end{itemize}
Here, $l_1$ and $l_2$ are chosen such that $b=b'+2l_1+l_2$.
Then the conditions of Lemma \ref{lem:unfolde} are satisfied, i.e.\ adding the equation $x=z-1$ reveals a path of edges of total length $b$ replacing the edge of length $b'$. \qedhere
\end{enumerate}
\end{proof}

  \begin{figure}[h]
    \begin{center}
      \begin{minipage}{0.4\linewidth}
        \centering
        \begin{tikzpicture}
          \draw
          (0,0) -- (4,0) -- (0,4) -- (0,0);
          \draw
          (0,3) -- (1,3)
          (0,3) -- (1,2)
          (0,2) -- (1,2)
          (0,0) -- (1,1)
          (1,0) -- (1,3)
          (1,0) -- (2,1)
          (1,2) -- (2,2)
          (1,1) -- (2,2)
          (1,0) -- (2,2)
          (2,0) -- (2,2)
          (2,1) -- (3,1)
          (2,1) -- (3,0)
          (3,0) -- (3,1);
          \fill (0,0) circle (2pt);
          \fill (1,0) circle (2pt);
          \fill (2,0) circle (2pt);
          \fill (3,0) circle (2pt);
          \fill (4,0) circle (2pt);
          \fill (0,1) circle (2pt);
          \fill (1,1) circle (2pt);
          \fill (2,1) circle (2pt);
          \fill (3,1) circle (2pt);
          \fill (0,2) circle (2pt);
          \fill (1,2) circle (2pt);
          \fill (2,2) circle (2pt);
          \fill (0,3) circle (2pt);
          \fill (1,3) circle (2pt);
          \fill (0,4) circle (2pt);
        \end{tikzpicture}
      \end{minipage}
      \begin{minipage}{0.5\linewidth}
        \centering
        \begin{tikzpicture}[scale=0.35]
          \draw[very thick,edgeRed] (0,0) -- (0,2) -- (1,3) -- (2,3) -- node[circle,inner sep=0.2mm,outer sep=0.2mm,fill=white] {$a$} (2,0);
          \draw (1,4) -- (1,3)
          (1,4) -- (0,4)
          (1,4) -- (2,5)
          (0,2) -- (-1,2)
          (2,3) -- (3,4);
          \draw[very thick,edgeGreen] (0,0) -- node[circle,inner sep=-0.5mm,outer sep=-0.5mm,fill=white] {$b'$} (2,0);
          \draw[very thick,edgeOrange] (0,0) -- node[circle,inner sep=0.2mm,outer sep=0.2mm,fill=white] {$d$} (3,-3) -- (5,-3);
          \draw (0,0) -- (-1,0)
          (3,-3) -- (3,-4);
          \draw[very thick,edgeBlue] (5,-3) -- node[circle,inner sep=0.2mm,outer sep=0.2mm,fill=white] {$c$} (2,0);
          \draw[very thick,edgeYellow] (5,-3) -- node[circle,inner sep=0.2mm,outer sep=0.2mm,fill=white] {$e$} (7,-4);
          \draw[very thick,edgeViolet] (7,-4) -- node[circle,inner sep=0.2mm,outer sep=0.2mm,fill=white] {$f$} (15,-4) -- (15,-6) -- (13,-8) -- (11,-8) -- cycle;
          \draw
          (16,-6) -- (17,-5)
          (16,-6) -- (16,-7)
          (16,-6) -- (15,-6)
          (15,-4) -- (16,-3)
          (11,-8) -- (11,-9)
          (13,-8) -- (13,-9);
        \end{tikzpicture}
      \end{minipage}
    \end{center}\vspace{-0.25cm}
    \caption{A curve of type \toneoneone{} with $c<d\leq b+c$.}
    \label{fig:111_2_1}
  \end{figure}

\section{Obstructions for realizably hyperelliptic curves}
\label{sec-hyperelliptic}

In this section, we prove Theorem~\ref{thm-main2}. The idea of the proof is the following: If $C\to \PP^n$ is a degree $4$ morphism from a smooth curve of genus $3$ that admits a faithful tropicalization $\Gamma\hookrightarrow\RR^n$ with $\Gamma$ being maximal, then the realizable sections of the canonical linear system of $\Gamma$ form a very ample linear system (Lemma~\ref{lem-canonical} below). However, Lemma~\ref{lem-nulltwonull}, Lemma~\ref{lem-oneoneone}, and Lemma~\ref{lemtwoonetwo} assert that this is not the case for {\em rh} curves $\Gamma$, which implies Theorem~\ref{thm-main2}. Moreover, for {\em rh} curves $\Gamma$ of types \tnulltwonull{} and \toneoneone{} we show that even the complete canonical linear systems are not very ample.

\begin{lemma}\label{lem-canonical} Let $C\to \PP^n$ be a degree $4$ morphism from a smooth curve of genus $3$ that admits a faithful tropicalization $\Gamma\hookrightarrow\RR^n$. Assume $b_1(\Gamma)>0$, e.g., $\Gamma$ is maximal. Then, $C\to \PP^n$ is an embedding given by a tuple of canonical sections. In particular, the realizable sections of the canonical linear system of $\Gamma$ separate points of $\Gamma$.
\end{lemma}

\begin{proof}
By the Riemann-Roch theorem, either the morphism $C\to\PP^n$ factors through $\PP^1$ or it is given by a tuple of sections of $\omega_C$. By our assumptions, the tropicalization $\Gamma\to \RR^n$ is an embedding and $b_1(\Gamma)>0$. Thus, $\Gamma\to \RR^n$ does not factor through the tropicalization of $\PP^1\to \PP^n$, and the assertion follows. In particular, since $\Gamma\hookrightarrow\RR^n$ is faithful, the tropicalizations of sections of $\omega_C$ separate points of $\Gamma$.
\end{proof}

Let $\Gamma$ be a trivalent tropical curve and $\chi\colon\Gamma\to\RR$ a piecewise linear function such that
\begin{equation}\label{eq:canineq}
K_\Gamma+div(\chi)\ge 0.
\end{equation}
The following simple observations will be useful in the proofs below:
\begin{itemize}
\item[(a)] $\chi$ restricted to any edge is convex, and hence achieves no maximum in the inner points of the edge unless it is constant on this edge;
\item[(b)] for any vertex where $\chi$ achieves its global maximum $M_\chi$, the slopes of $\chi$ vanish along at least two of the attached edges due to condition \eqref{eq:canineq}. Moreover, $\chi$ must be constant on any such edge by convexity.
\end{itemize}
For a piecewise linear function $\chi\colon\Gamma\to\RR$, we denote the slope of $\chi$ at a vertex $V$ along an attached edge $E$ by $\frac{\partial\chi}{\partial E}(V)$.

\begin{lemma}\label{lem-nulltwonull} The canonical divisor of an abstract tropical curve $\Gamma$ of type \tnulltwonull{}, where the edges forming the $2$-cut have equal lengths, does not separate points of $\Gamma$. 
\end{lemma}
\begin{proof}
Denote the vertices of $\Gamma$ by $A,B,C,D$ and the edges by $E_a,E_b,E_c,E_d,E_e,E_f$ as in Figure~\ref{fig:type020}. The length of $E_i$ is $i$, and we have $c=d$.
Let $\chi\colon\Gamma\to\RR$ be a piecewise linear function satisfying \eqref{eq:canineq}. We shall show that $\chi(A)=\chi(B)$ and $\chi(C)=\chi(D)$.

Assume without loss of generality that $\chi(A)=M_\chi$. Then $\chi(B)=\chi(A)$ by (b). If $\chi$ is constant on either $E_c$ or $E_d$ then by the same argument $\chi(C)=M_\chi=\chi(D)$. Otherwise $\frac{\partial\chi}{\partial E_c}(A)=\frac{\partial\chi}{\partial E_d}(B)=-1$ by (b) and \eqref{eq:canineq}. Assume to the contrary that $\chi(C)\ne\chi(D)$, say $\chi(C)>\chi(D)$. By (a), $\frac{\partial\chi}{\partial E_e}(C), \frac{\partial\chi}{\partial E_f}(C)<0$, and hence $\frac{\partial\chi}{\partial E_c}(C)\ge 1$ by condition \eqref{eq:canineq}. But $\chi$ is convex on $E_c$ and hence linear. Thus, by convexity of $\chi$ on $E_d$, we obtain $\chi(D)\ge \chi(B)-d=\chi(A)-c=\chi(C)$, which is a contradiction.
\end{proof}

\begin{lemma}\label{lem-oneoneone} The canonical divisor of an abstract tropical curve $\Gamma$ of type \toneoneone{}, where the edges forming the $2$-cut have equal lengths, does not separate points of $\Gamma$.
\end{lemma}
\begin{proof} The argument is similar to the previous proof. Let $\Gamma$ have vertices $A$, $B$, $C$ and $D$ and edges $E_i$ of length $i$ for $i=a,\ldots,f$ as in Figure~\ref{fig:type111}.

Let $\chi\colon\Gamma\to\RR$ be a piecewise linear function satisfying \eqref{eq:canineq}. We shall show that $\chi(A)=\chi(B)$.
First, notice that by convexity, the slopes of $\chi$ along the loop at $D$ are non-positive, and hence $\frac{\partial\chi}{\partial E_e}(D)\ge -1$. Thus, again by convexity, $\frac{\partial\chi}{\partial E_e}(C)\le 1$. Second, assume to the contrary that $\chi(A)\ne\chi(B)$, say $\chi(A)<\chi(B)$. Then $\frac{\partial\chi}{\partial E_a}(B),\frac{\partial\chi}{\partial E_b}(B)<0$ by (a), and hence $\frac{\partial\chi}{\partial E_d}(B)\ge 1$. By convexity this implies $\chi(C)>\chi(B)>\chi(A)$. Thus, again by convexity, $\frac{\partial\chi}{\partial E_d}(C),\frac{\partial\chi}{\partial E_c}(C)\le -1$, and hence $\frac{\partial\chi}{\partial E_d}(C)=\frac{\partial\chi}{\partial E_c}(C)=-1$ since the sum of the three slopes at $C$ must be at least $-1$. We conclude that $\chi$ is linear along $E_d$. Finally, $\chi(B)=\chi(C)-d=\chi(C)-c\le\chi(A)$ by convexity of $\chi$ restricted to $E_c$, which is a contradiction.
\end{proof}


\begin{remark}\label{rem-twoonetwo} For $\Gamma$ of type \ttwoonetwo{}, the canonical divisor is very ample, even if the lengths of the banana edges are equal. It is clear that we can separate points which are not on the banana edges, of the same distance to the vertices.
For such two points, we can use functions as depicted in Figure~\ref{fig:functionon212}.

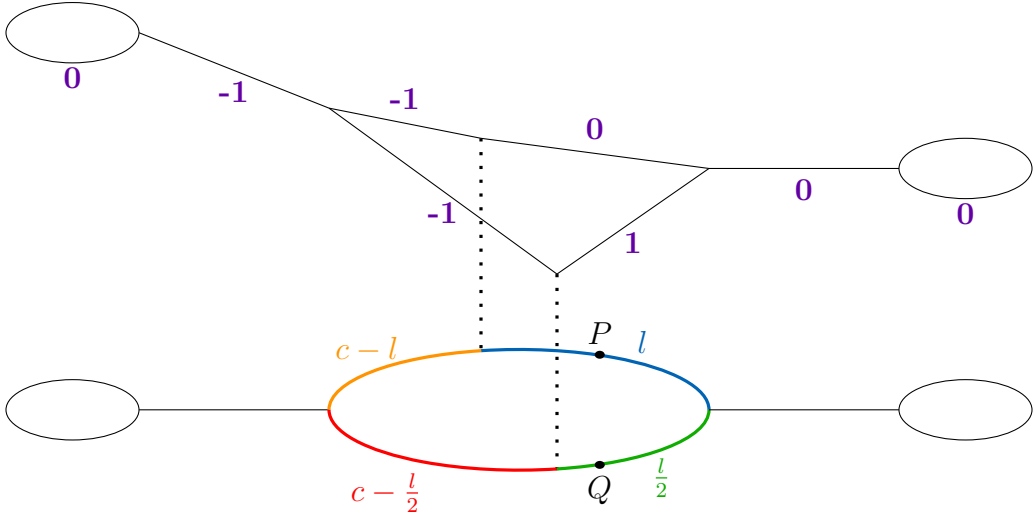
\begin{figure}[ht]
  \begin{center}
    \begin{tikzpicture}[xscale=2.5,yscale=2]
      \draw (-2.35,2.5) ellipse (0.35 and 0.2);
      \node[edgeViolet] at (-2.35,2.2) {\textbf{0}};
      \draw (-2,2.5) -- node[below,edgeViolet] {\textbf{-1}} (-1,2);
      \draw (-1,2) -- node[above,edgeViolet,yshift=0.5mm] {\textbf{-1}} (-0.2,1.8);
      \draw (-0.2,1.8) -- node[above,edgeViolet,yshift=0.5mm] {\textbf{0}} (1,1.6);
      \draw (-1,2) -- node[below,edgeViolet] {\textbf{-1}} (0.2,0.9);
      \draw (0.2,0.9) -- node[below,edgeViolet] {\textbf{1}} (1,1.6);
      \draw (1,1.6) -- node[below,edgeViolet] {\textbf{0}} (2,1.6);
      \draw (2.35,1.6) ellipse (0.35 and 0.2);
      \node[edgeViolet] at (2.35,1.3) {\textbf{0}};

      \draw[loosely dotted, very thick] (-0.2,1.8) -- (-0.2,0.4);
      \draw[loosely dotted, very thick] (0.2,0.9) -- (0.2,-0.4);

      \draw (-2.35,0) ellipse (0.35 and 0.2);
      \draw (-2,0) -- (-1,0);
      \draw[very thick,edgeBlue] (1,0) arc (0:101.5:1 and 0.4) node (A) {};
      \node[edgeBlue] at (0.65,0.45) {$l$};
      \draw[very thick,edgeOrange] (A.center) arc (101.5:180:1 and 0.4);
      \node[edgeOrange] at (-0.8,0.4) {$c-l$};
      \fill (0.425,0.365) circle (0.75pt);
      \node[above] at (0.425,0.365) {$P$};
      \draw[very thick,edgeRed] (-1,0) arc (180:281.5:1 and 0.4) node (B) {};
      \node[edgeRed] at (-0.7,-0.55) {$c-\frac{l}{2}$};
      \draw[very thick,edgeGreen] (B.center) arc (281.5:360:1 and 0.4);
      \node[edgeGreen] at (0.75,-0.45) {$\frac{l}{2}$};
      \fill (0.425,-0.365) circle (0.75pt);
      \node[below] at (0.425,-0.365) {$Q$};
      \draw (1,0) -- (2,0);
      \draw (2.35,0) ellipse (0.35 and 0.2);

    \end{tikzpicture}\vspace{-0.5cm}
  \end{center}
  \caption{A function in the linear system of the canonical divisor, separating two points $P$ and $Q$. The numbers above indicate edge slopes while the numbers below indicate edge lengths.}
  \label{fig:functionon212}
\end{figure}

\end{remark}

Notice that \cite[Theorem~10]{Amini} classifies all graphs with a not very ample canonical divisor, but this result seems to contain a small gap. For that reason, we include the arguments used in genus $3$ here.
%

\begin{lemma}\label{lemtwoonetwo} Let $\Gamma$ be a tropical curve of type \ttwoonetwo{} such that the banana edges have identical lengths. Then any realizable canonical divisor $K_\Gamma+div(\chi)$ satisfies $\chi|_{E_c}=\chi|_{E_d}$. In particular, points on the banana edges cannot be separated with realizable sections.
\end{lemma}
\begin{proof}
Without loss of generality $\chi(B)\ge\chi(C)$. Thus, $\frac{\partial\chi}{\partial E_c}(B), \frac{\partial\chi}{\partial E_d}(B)\le 0$ by convexity. If one of the slopes vanishes then $\chi$ is constant on both edges by convexity and condition (ii) in \cite[Theorem~6.3]{MUW}. Assume now that both slopes are negative. An argument identical to the one used in the proof of Lemma~\ref{lem-oneoneone}, shows that $\frac{\partial\chi}{\partial E_c}(B)+\frac{\partial\chi}{\partial E_d}(B)\ge -2$, and hence $\frac{\partial\chi}{\partial E_c}(B)=\frac{\partial\chi}{\partial E_d}(B)=-1$. Similarly,
\begin{equation}\label{eq:212ineq}
\frac{\partial\chi}{\partial E_c}(C)+\frac{\partial\chi}{\partial E_d}(C)\ge -2.
\end{equation}
Let us identify $E_c\simeq [0,c]$ and $E_d\simeq [0,d]$ such that $B$ is identified with $0$. Then there exist $0=t_0\le t_1\le t_2\le t_3\le t_4=c$ and $0=s_0\le s_1\le s_2\le s_3\le s_4=d$ such that $\frac{\partial\chi}{\partial x}|_{(t_i,t_{i+1})}=\frac{\partial\chi}{\partial x}|_{(s_i,s_{i+1})}=i-1$. Thus, $\chi(C)=2t_4-t_1-t_2-t_3=2s_4-s_1-s_2-s_3$, and since $c=d$, we obtain
\begin{equation}\label{eq:212eq}
t_1+t_2+t_3=s_1+s_2+s_3.
\end{equation}
By the symmetry of $\Gamma$ we may assume that $t_3\le s_3$. If $t_3<s_3$ then $\frac{\partial\chi}{\partial E_c}(C)=-2$ and $\frac{\partial\chi}{\partial E_d}(C)\ge 0$ by \eqref{eq:212ineq}. This implies that $s_2=s_3=s_4=d$, and hence $t_1>s_1$ by \eqref{eq:212eq}. But this contradicts condition (ii) in \cite[Theorem~6.3]{MUW}. Thus, $t_3=s_3$.

Similarly we may assume that $t_1\le s_1$. If $t_1<s_1$ then $t_2>s_2\ge s_1>t_1$ by \eqref{eq:212eq}, which again contradicts condition (ii) in \cite[Theorem~6.3]{MUW}. Thus, $t_1=s_1$, and hence also $t_2=s_2$ by \eqref{eq:212eq}, which completes the proof.
\end{proof}

\begin{proof}[Proof of Theorem~\ref{thm-main2}]
Let $\Gamma$ be a maximal {\em rh} tropical curve of genus $3$. Then, as described in Section~\ref{subsec:moduliandcurves}, $\Gamma$ is of type \tnulltwonull{}, \toneoneone{} or \ttwoonetwo{}, where in each case the edges forming the $2$-cut have equal lengths. In each of the three possible cases the linear system of realizable sections does not separate points on $\Gamma$ by one of the Lemmas~\ref{lem-nulltwonull},~\ref{lem-oneoneone},~\ref{lemtwoonetwo}. Thus, the assertion of Theorem~\ref{thm-main2} follows from Lemma~\ref{lem-canonical}.
\end{proof}




\bibliographystyle{plain}

\end{document}